\documentclass[11pt]{article}
\usepackage[
  a4paper,
  margin=0.7in,
  headsep=5pt, % separation between header rule and text
]{geometry}
\usepackage{macros}
\usepackage{xcolor}
\usepackage{color}
\usepackage{microtype}
\usepackage{epigraph}
\usepackage[autostyle]{csquotes} 
\usepackage{newpxtext}
\usepackage{graphicx}
\graphicspath{ {./images/} }

\title{G\"odel's Incompleteness after Joyal}

\begin{document}
\author{Joost van Dijk, Alexander Gietelink Oldenziel}
\maketitle
\pagenumbering{arabic}
\begin{abstract}
We give proofs of G\"odel's incompleteness theorems after Joyal. The proof uses internal category theory in an arithmetic universe, a predicative generalisation of topoi. Applications to L\"ob's Theorem are discussed. 
\end{abstract}
\tableofcontents
%%%%

%\include{Chapters/ClassifyingAU}
%\include{Chapters/Subspaces}
\section{Introduction}
This article concerns an alternative proof of the first G\"odel Incompleteness theorem in the language of category theory due to André Joyal, relying crucially on his newly introduced notion of 'Arithmetic Universe'.\\
In 1973 Joyal lectured on his new proof, and a set of notes were circulated among a small group of workers in topos theory.  Unfortunately, the proof has never been made publicly available. This document means to remedy this gap in the literature. Additionally, we will try to indicate how the proof using arithmetic universes is related to the traditional proof. We will also detail a version of Lob's theorem.
\bigbreak \noindent
The G\"odel Incompleteness theorem remains without a doubt one of the high points of 20th century Mathematics. G\"odel's brilliant insight was the notion of arithmetication, simultaneously straightforward and profound. Straightforward, for the construction is a simple if occasionally finicky exercise in encoding various symbols and axioms, yet profound for it allows \emph{formal systems of mathematics to speak about themselves}. This self-reflective ability of formal systems has become a defining feature of many later advances in mathematical logic.\\
As mentioned, Joyal's alternative proof is based chiefly on the notion of 'Arithmetic Universe' a generalisation of the notion of a topos. Joyal constructs the initial Arithmetic Universe $\mathcal{U}$ using ideas from primitive recursive arithmetic. This construction is explicit and may be repeated within the Arithmetic Universe $U_0$. This uses the fact that Arithmetic Universes have enough structure to interpret many constructions that may be performed in Set. A more refined treatment would say that the internal language of $\mathcal{U}$ is a sufficiently expressive type theory - we will say more on this later. The rest of the argument then relies on manipulating a categorical incarnation of the Godel sentence constructed from $\mathbb{U}$ in $ \mathcal{U}$. 
\begin{defn}
An \emph{arithmetic universe} is a list-arithmetic pretopos. That is, a category with finite limits, stable disjoint coproducts, stable effective quotients by monic equivalence relations and parameterized list-objects. 
\end{defn}
 Quite a mouthful! Roughly, the first two conditions will allow us to do a large number of basic mathematical constructions; in the internal language it provides for conjunction $\wedge$, disjunction $\vee$, truth $\top$ and the falsum $\bot$. The third condition allows for quotients by equivalence relations; logically it adjoins an existential operator $\exists$.
\footnote{The expert reader will have noticed the similarity between the second and third of these axioms and the Giraud axioms of topos theory.}
The final condition, requiring list-objects will allow for the use of (primitive) \emph{recursion}. 
The explicit construction by André Joyal of the initial model $U_0$ of arithmetic Universes uses techniques of primitive recursive arithmetic, bringing in the ability to 'code' various mathematical objects. \\
The category $\mathcal{U}$ is build in stages. First, one starts with the initial 'Skolem theory' $\Sigma_0$, in effect a category whose objects are all products of $\nno$. This rather simple category has an internal language, that corresponds to register machines: a programming language where one has an infinite amount of memory states, and one can increase them by 1, set them to 0, or loop over a command. The resulting system is equivalent to that of primitive recursive functions. \\
The next step is to consider the category of decidable predicates in $\Sigma_0$, denoted $\Pred(\Sigma_0)$. The final step adjoins quotients to obtain $\Pred(\Sigma_0)_{ex/reg}$. It will be a theorem that this coincides with the initial arithmetic universe $\mathcal{U}_0$. 
The initiality of $\mathcal{U}_0$ is key, for it implies that $
\mathcal{U}_0$ is the syntactic category $Syn(\mathcal{T}_{AU})$, where $\mathcal{T}_{AU}$ is a weak variant on Martin-Lof type theory. 

\textit{Acknowledgements}. The authors would like to thank  Andr\'e Joyal, Antoine Beaudet, Ingo Blechschmidt, Gavin Wraith, David Roberts, Paul Levy, Steve Vickers, Paul Taylor, Alan Morrison and Sina Hazratpour for their help. 

\section{Arithmetic Universes}
In this section we cover some of the basics surrounding arithmetic universes.
\begin{defn}
Let $A,B$ be two objects in a category $\mathcal{C}$ with coproducts and pullbacks. We may form the coproduct $A+B$. The category $\mathcal{C}$ is said to have \emph{disjoint coproducts} if the pullback
\[
\begin{tikzcd}
   A\times_{A+B}B \arrow[r] \arrow[d]& B \arrow[d] \\
   A \arrow[r] & A+B  
\end{tikzcd}
\]
is isomorphic to the initial object $0$ [coproduct over the empty diagram].
\end{defn}
\begin{defn} Let $C$ be a category equipped with a terminal object.
A \emph{natural numbers object} is an object $\N$ equipped with a map $0: 1 \to \N$ and a map $S: \N \to \N$ such that for every $a: 1 \to X$ and $g: X \to X$ there is a unique map $f: \N \to X$ such that the following diagram commutes:
\[
\begin{tikzcd}
   1 \ar[r, "0"] \ar[dr, "a", swap] & \N \ar[r, "S"] \ar[d, "f"] & \N \ar[d, "f"] \\
 & X & X \ar[l, "g"] 
\end{tikzcd}
\]
\end{defn}
\begin{defn} \label{recursion}
Let $\mathcal{C}$ be a category with finite products. A \emph{parametrized natural numbers object} is an object $\N$ equipped with a map $0: 1 \to \N$ and a map $S: \N \to \N$ such that for every $a: A \to X$ and $g: X \to X$ there is a unique map $f: A \times \N \to X$ such that the following diagram commutes:
\[
\begin{tikzcd}
   A \ar[r, "0"] \ar[dr, "a", swap] & A \times \N \ar[r, "Id \times S"] \ar[d, "f"] & A \times \N \ar[d, "f"] \\
 & X & X \ar[l, "g"] 
\end{tikzcd}
\]
\end{defn}
\begin{defn}
A category $\mathcal{C}$ equipped with finite limits has \emph{parameterized list objects} if for any object $A \in Ob \mathcal{C}$ there is an object $List(A)$ with morphisms $c: 1\to List(A)$ and $app_A: List(A) \times A  \to List(A)$ such that for every $b:B \to Y$ and $g: Y\times A \to A$ there is a unique $rec(b,g)$ making the following diagram commute
\[
\begin{tikzcd}
   B\arrow[dr,"b"] \arrow[r,"{(id_B,c)}"] &B\times List(A) \arrow[d,"{rec(b,g)}"]& B\times (List(A)\times A) \arrow[l, "{id_B\times app_A}"] \arrow[d,"{rec(b,g)} \times id_A"]\\
   &Y&Y\times A \arrow[l,"g"]
\end{tikzcd}
\]
\end{defn}
\begin{remark}
The need for parameterized natural number objects and parameterized list objects is because as we will see arithmetic universe will not be Cartesian closed. In this setting the non-parameterized versions are not well-behaved. 
\end{remark}
\begin{defn} Let $C$ be a category with finite limits. An internal equivalence relation on $X$ is a subobject $R \hookrightarrow X \times X$ equipped with the following morphisms:
\begin{itemize}
    \item(Reflexivity) $r:X \to R$ which is a section of $p_1: X\times X \to X$ and of $p_2: X\times X \to X$.
    \item(Symmetry) $s: R\to R$ such that $p_1\circ s=p_2$ and $P_2 \circ s=p_1$. 
    \item(Transitivity) $t: R\times_X R \to R$ where if
\[
\begin{tikzcd}
   R\times_XR \arrow[r,"q_2"] \arrow[d, "q_1"] & R \arrow[d, "p_1"] \\
   R \arrow[r, "p_2"] & X
\end{tikzcd}
\]
is the pullback square the equations $p_1=\pi_1 \circ i, p_2=\pi_2 \circ i, p_1 \circ q_1=p_1 \circ t$ and $p_2\circ q_2=p_2\circ t$. Here
\[
\begin{tikzcd}
   R \arrow[r, hook, "i"] & X\times X \arrow[r, "\pi_1", shift right=1ex] \arrow[r,shift left=1ex, "\pi_2"]&  X 
\end{tikzcd}
\]
\end{itemize}
\end{defn}
\begin{defn}The coequalizer $X/R:= coeq(i_1,i_2)$ of an internal relation $R \hookrightarrow X\times X$ is called a quotient object:
\[
\begin{tikzcd}
   R  \arrow[r,"i_1", shift right=1ex] \arrow[r,"i_2", shift left=1ex]&  X \arrow{r}& X/R 
\end{tikzcd}
\]
A quotient $X\to X/R$ is called \emph{effective} when it arises as a kernel pair, i.e. the monomorphism
\[
X\times_{X/R}X \to X
\]
\end{defn}
\begin{defn}
Let $C$ be a category with finite limits. Let $P$ be a (categorical) property of an object or diagram of objects $I$ in $C$. We say $P$ is stable if when $P$ holds for $I$ it also holds for the pullback $f^*(I)$ in $C/A$ for any $f: A \to 1$.
\end{defn}
\begin{defn}
A \emph{pretopos} is a category equipped with finite limits, stable finite
disjoint coproducts and stable effective quotients of monic equivalence relations.
An \emph{arithmetic universe} is a pretopos which has parametrized list objects. 
\end{defn}

The theory of arithmetic universes is highly reminiscent of that of topoi. Yet arithmetic universes are quite different from topoi in a number of ways.
\begin{rmk}[Small versus large] Typically, almost every topos one deals with is large. In contrast in AU-theory there are many interesting small AU's. This has the distinct advantage that we may describe internal AU's simply as certain internal categories, unlike the case of (Grothendieck) topoi where one has to resort to indexed categories. \\
Small indexed categories and internal categories are almost equivalent, except that internal categories are more strict, various equations holding up to equality instead of isomorphism. Strictness is often important in obtaining interpretations for various (type-theoretic) languages. 
Explicit manipulation of internal categories will be key in the proof of G\'odel's incompleteness theorem.
\end{rmk}
\begin{rmk}[Recursive versus arbitrary infinities]
An important topic in topos theory is that of geometric logic and geometric theories. A distinct feature is the ability to consider infinitary theories over a topos $\mathcal{E}$, where axioms can be build as $\vee_{i\in I} A_i$-operator of atomic sentence $A_i$ over infinite collections $I$. These infinities arise from the base topos $\mathcal{E}$, in the sense that $I$ may be any object of $\mathcal{E}$. In the arithmetic approach, one cannot index over 'any' infinite set. Instead, one has to give an explicit recursive description of these sets. This phenomenon was the central impetus for investigating 'arithmetic reasoning' [where infinities must be recursively described] instead of 'geometric reasoning' [where the infinities may come from the base], see \cite{hazratpour18} and \cite{vickers17} based principally on Vicker's geometrization programme already outlined in the last section of \cite{vickers99a}.
\end{rmk}
\begin{rmk}[Predicative versus impredicative reasoning]
A third difference between topos theory and AU-theory is that of predicativity. Traditionally a topos is assumed to be locally cartesian closed, equipped with power objects and dependent products. Power sets have been regarded with suspicion, going back all the way back to Russel and Poincare, for they allow for 'impredicative reasoning', a fairly imprecise term referring to definitions of objects $X$ that already implicitly assume the existence of said object $X$. A famous instance of how \emph{impredicative reasoning} can be treacherous is Berry's paradox: 'the smallest positive integer not nameable in under 60 letters'. There are only finitely many strings under 60 letters, but there are infinitely many numbers. Hence there must be a least number not nameable in 60 letters. But we have just named it!\\
Russel argued that the paradoxical nature of Berry's non-nameable number is caused by a form of \emph{vicious circularity}. What goes wrong, he claimed, is that an entity is defined, or a proposition is formulated, in a way that is dangerously circular. \footnote{Not all mathematicians conclude from Berry's paradox the unreliability of impredicative reasoning. A more conventional view isolates the vagueness in the notion of 'definability' or 'nameability' as the culprit.} Nowadays, issues of predicativity have been mostly forgotten and even constructive mathematicians use power sets with gay abandon. Yet at the start of the 20th century the issue of impredicative reasoning attracted the attention of big-name mathematicians, the first predicative foundations going back to Weyl.\\
The arithmetic reasoning that is legitimate in AU's is natively predicative. The first important instance where predicative and impredicative foundations diverge occurs in the definition of topological space. Traditionally a topological space is a set together with a collection of subsets, ostensibly relying fundamentally on the power set axiom of ZFC.  In predicative foundations we need to change the definition of topological space; instead of regarding open sets as fundamental we regard a generating basis of open sets as the fundamental and axiomatises these, this leads to the notion of [various forms of] \emph{formal spaces}. 
Let us consider a concrete example where this change of definition will change the way we think about 'spaces'. The closure $cl(T)$ of a subset of $T$ of a topological space $X$ can be traditionally defined as the intersection $\bigcap_{T \subset C} C$ of all closed sets containing $T$. This is problematic predicatively as we quantify over a set $\{T \subset C\}$ to define $cl(T)$ that itself includes $cl(T)$.
In formal topology [which crop up when one considers the analogy of subtoposes for Arithmetic Universes] the notion of closure splits and closed sets split in the sense that there are multiple notions that are inequivalent predicatively, yet equivalent when one allows for impredicative notions like unrestricted power sets. 
\\
\end{rmk}

%%%%%
\section{Primitive Recursive Arithmetic and the Initial Arithmetic Universe}
\label{chapter:nno} 
\begin{defn}
A Skolem category $C$ is a category with finite products equipped with a natural number object. A \emph{Skolem theory} $\Sigma$ is a Skolem category where every object is a finite product of the natural number object $\nno_{\Sigma}$. 
A morphism of Skolem Theories is a functor which preserves finite products and the Natural Number object.
\end{defn}
 A Skolem theory is a kind of Lawvere theory where the operations arise from zero, successor and recursion. 

\begin{example}
The category of Skolem Theories has an intial object $\Sigma_0$ which might be called the minimal Skolem theory, or the "free theory"  generated by the basic data.
\end{example}
\begin{caveat}
One cannot identify the arrows of $\Sigma_0$ with actual functions; for let $\Sigma_{standard}$ be the full subcategory of $Set$ with objects $1, \mathbb{N}, \mathbb{N}^2,\cdots etc$. Since $\Sigma_0$ is initial we have a map of Skolem Theories:
\[
\Sigma_0 \to \Sigma_{standard}
\]
\[
(f:\mathbb{N}^k \to \mathbb{N}^r) \mapsto (|f|:\mathbb{N}^k \to \mathbb{N}^r)
\]
the idea is that $f$ exists as a primitive recursive algorithm for the 'actual' set-theoretic function $|f|$, and it can happen that $|f|=|g|$ but $f\not = g$. This might sound a little mysterious if one thinks in terms of the set-theoretic conception of functions, which identifies functions with their graph. We want to work in a setting where functions are more akin to algorithms; it may happen that functions coincide \emph{extensionally} but not \emph{intensionally}. \\
Extensional equality of functions $|f|=|g|$ means simply that $f(x)=g(x)$ for all $x$ in the domain, but functions might have algorithmic differences which prevent the $f$ from being equal to $g$. Intensional character of a function is the additional features of a function, such as its algorithmic character, that go beyond the extension (or graph) $|f|$ of the function $f$. Another way of seeing this is that two functions $f,g$ might be given by different algorithms and even if in Set theory we may prove $f(x)=g(x)$ for all $x$ in the domain, there might not be any primitive recursive construction that witnesses this! 
\end{caveat}
\begin{example} Let $a:\Sigma_0\to \Sigma_{Standard}$ be the canonical map. Consider the category obtained by inverting all maps $f:X\to Y \in \Sigma_0$ for which $a(f)$ is a bijection in $\Sigma_{Standard}$. By a theorem of Kleene this is the Skolem theory of total recursive functions.
\end{example}
A procedure will now be described which completes any Skolem Theory to a category with finite limits. Let $\Sigma$ be a Skolem theory. The idea is to adjoin "decidable subsets": define a subset of $\mathbb{N}^k$ to be a map $x: \mathbb{N}^k \to \mathbb{N} $ such that $x \wedge 1=x$. Write $\mathcal{P}_{dec}( \mathbb{N}^k ) $  for this class of subsets, Then $\mathcal{P}_{dec}$ is actually a contravariant functor: 
\[
(\mathbb{N}^k \xrightarrow{f} \mathbb{N}^r) \quad \text{induces} \quad \mathcal{P}_{dec}(\mathbb{N}^r )\xrightarrow{f^{\ast}} \mathcal{P}_{dec}(\mathbb{N}^k) 
\]
The diagonal $\Delta \in \mathcal{P}_{dec}(\mathbb{N}) $ is given by $E:\nno^2 \to \nno$ defined above; similarly one defines a diagonal $\Delta_k \in \mathcal{P}_{dec}(\nno^{2k}) $, for all $k$. The class of subsets $\mathcal{P}_{dec}(\nno^k)$ actually carries a Boolean algebra structure: unions, intersections and complements of subsets correspond to the supremum, infinum and 1 minus the corresponding characteristic functions. 

\begin{defn} Given a Skolem theory $\Sigma$ the category $\Pred(\Sigma)$ of
predicates in $\Sigma$ is defined as follows:
\begin{itemize}
    \item $Ob(\Pred(\Sigma)):$ Decidable predicates in $\Sigma$, that is morphisms $P: \nno \to \nno$ such that $P\ast P=P $.
    \item $Mor(\Pred(\Sigma))$: $\Sigma$-morphisms $f: \nno \to \nno$ such that $P \leq Q \circ f$ [where $Q\circ f$ is the composition of $Q$ with $f$] and two such $\Sigma$-morphisms $f,g : \nno \to \nno$ are equal if $P \ast f = P \ast g$
\end{itemize}
\end{defn}
\begin{prop}
Here are some properties of $\Sigma \to \Pred(\Sigma)$:
\begin{enumerate}
\item $\Pred(\Sigma)$ has finite limits: for example the equalizer of $f,g: \nno^k \toto \nno^r $ is just the subset $S\in \mathcal{P}_{dec}(\nno^k)$ given by $\nno^k \xrightarrow{(f,g)} \nno^r \times \nno^r \xrightarrow{E} \nno$.
\item $\Pred(\Sigma)$ is regular and satisfies the axiom of choice. In fact, any arrow factors as a split epi followed by a mono  ( caution: do not confuse monomorphisms with subsets!). 
Suppose we are given $f: \nno \to \nno $: we may factor $f$ as $\nno \twoheadrightarrow im(f) \hookrightarrow \nno$. Define $im(f)=\{<f(n),m>\in \nno \times \nno | m=\mu(k \leq n\in f^{-1}(n)) \} $, which is has the splitting of $\nno \twoheadrightarrow im(f)$ given by $(f(n),m> \mapsto m$.
Here the $\mu$ operator is primitive recursively defined as bounded minimization: it looks for the minimum $k\in f^{-1}(n)$ over all natural numbers smaller than $n$ which is indeed primitive recursive [unbounded minimisation is of course not primitive recursive]. 
\item Coproducts exist in $\Pred(\Sigma)$: given decidable subsets $S,T$ of $\nno$, for example, take the union of the decidable subsets $S\times \{0\}$ and $T\times \{1\}$ of $\nno^2$.
\item for any graph object
\[
G \to A_0
\]
in $\Sigma$, there exists a free category object $A_1 \toto A_0 $ over $G$ in $\hat{\Sigma}$
\end{enumerate}
\end{prop}
\begin{proof}
See Proposition 4.7 of \cite{maietti10a}.
\end{proof}
\begin{example}
The category $Pred(\Sigma_0)$ is the
category of decidable primitive recursive predicates.
\end{example}
For our purposes $\Pred(\Sigma)$ does not have enough categorical properties. It is necessary to make a second completion $\Sigma \to \Pred(\Sigma) \to \exact{\Pred(\Sigma)}$ which adds quotients. 
\begin{defn}[Exact/Regular completion]
Let $\mathcal{C}$ be a regular category. We form a new category whose objects are pairs $(R,A)$ with $R \hookrightarrow A\times A$ an equivalence relation on $A$, and whose maps $(R,A) \to (R,B)$ are classes of maps $f:A \to B$ such that $R \leq (f \times f)^{-1}(Q)$, under the relation $f \sim g$ if and only if there exists a lifting of $(f,g)$: 
\[
\begin{tikzcd}
& Q \arrow[d]\\
A \arrow[r] \arrow[ur,dotted] & B \times B
\end{tikzcd}
\]
The resulting category is an exact category. 
\end{defn}

\begin{prop}The category $(\Pred(\Sigma_0))_{ex/reg}$ is an arithmetic universe. In particular the following properties hold: 
\begin{itemize}
    \item finite limits exists
    \item it is a regular category [but notice that the step $\Pred(\Sigma) \to \exact{\Pred(\Sigma)}$ spoils the splitting of the image factorisation] 
    \item coproducts exists.
    \item free category objects exist for any graph object
    \item quotients exists (but are not split in general) 
\end{itemize}
\end{prop} 
\begin{proof}
See proposition 4.10 of \cite{maietti10a}.
\end{proof}
These axioms are the defining properties for the notion for Arithmetic Universe (AU). Applying the completion procedure for $Pred(\Sigma_0)$ yields the initial Arithmetic Universe $U_0$.
\begin{thm} The category $\exact{\Pred(\Sigma)}$ coincides with the initial arithmetic universe. 
\end{thm}
\begin{proof}
See Theorem 6.2 of \cite{maietti10a}.
\end{proof}

\section{Arithmetic Type Theory}

We will see that there is a very precise correspondence between the initial arithmetic universe $U_0$ and the arithmetic type theory calculus $\mathcal{T}_{AU}$. This typed calculus can be thought of as a type-theoretic incarnation of primitive recursive arithmetic. Indeed, any primitive recursive function can be encoded in this type theory.
There is also a related but different correspondence between register machines and Skolem theories, cf \cite{morrison96}. Using register machines one can also encode primitive recursive functions. 

In this chapter we will introduce arithmetic type theory and detail how it corresponds to the internal language of arithmetic universes. 

\begin{rmk}
Type theory might not be familiar to most mathematicians brought up in the set-theoretic tradition. Type theory is an alternative foundation for mathematics, whose basic objects are not elements and sets as in traditional set-theory but terms and types. Proponents claim it to have several distinct advantages over the  traditional set-theory based formulation of mathematics. For one, Type Theory is natively constructive. It also integrates the logic and the rest of the foundational system more tightly than in the traditional set-up, where ZFC is build on top of a specified logical syntax. Interest in type theory has surged with the recent development of Homotopy Type Theory, see \cite{HoTT2013}. We will be mostly concerned with a specific type theory called Arithmetic Type theory $\mathcal{T}_{AU}$. 
\end{rmk}

\begin{rmk}
Objects of study of type theory, i.e. types, have different ontological status than objects of study of set theory, i.e. sets. Types are constructed together with their elements, and not by collecting some previously existing elements unlike the case of sets. ``A type is defined by prescribing what we have to do in order to construct an object of that type.''~\cite{pml:int-th-type}.
\iffalse
 For example, from a constructive point of view, in order to give a Cauchy real number, we have to give a sequence of rational numbers together with a proof that this sequence satisfies Cauchy condition. Thus, the type of Cauchy real numbers is
\[
\mathbb{R}_c \eqdef \sum_{\substack{x\colon\mathbb{N} \to \mathbb{Q}}} \ \prod_{m:\mathbb{N}} \ \prod_{n:\mathbb{N}} \ (|x_{m+n} - x_m| \leq 2^{-m}) 
\] 
\fi
The fundamental principle of type theory is that types should be defined by introduction, elimination, and computation rules.
This is closely related to the well-known principle of category theory: objects should be defined by universal properties.

To make this point clear we give the example of binary product as a universal construction in category theory. The following table illustrates the connection between categorical products and type theoretic products:
\[
\begin{tabular}{c@{\hskip 1in}c@{\hskip 0.5in}c}    \toprule
\emph{Type theory}     & \emph{Category theory}     \\ \midrule
  $\inferrule{z:A \times B}{\fst{z}:A,\ \snd{z}:B}$ & $A \xlaw{\pi_1} A \times B \xraw{\pi_2} B $    \\ 
  & \\
$\inferrule{a:A,\ b:B}{\tuple{a}{b}:A \times B}$ & 
\begin{tikzpicture}[baseline=(a.base)]
\node (a) at (0,0) {$X$};
\node (b) at (2.5,0) {$A \times B$};
\node (c) at (4.3,1) {$A$};
\node (d) at (4.3,-1) {$B$};
\draw[->]
(b) to node[midway,fill=white]{$\pi_1$} (c);
\draw[->]
(b) to node[midway,fill=white]{$\pi_2$} (d);
\draw[->]
(a) to[bend left=30] node(e_3){} node[midway,fill=white]{$a$} (c);
\draw[->]
(a) to[bend right=30] node(e_4){} node[midway,fill=white]{$b$} (d);
\draw[->, dashed]
(a) to node[midway,above]{$\tuple{a}{b}$} (b); 
\end{tikzpicture} \\
 &  \\
$\fst \tuple{a}{b} = a$ & $\pi_1  \tuple{a}{b} = a$  \\  
$\snd \tuple{a}{b} = a$ & $\pi_2  \tuple{a}{b} = b$   \\
& \\
$\tuple{\snd{z}}{\fst{z}} = z$ & uniqueness (in the UP)  \\
\bottomrule
 \hline
\end{tabular}
\]
\end{rmk}
Type theory is an alternative to the traditional set-theoretic foundations. Recently, a variant called Homotopy type theory has attracted a great amount of interest [\cite{HoTT2013}]. Formerly the beau of a small cadre of logicians, computer scientists and heterodox mathematicians type theory has blossomed with the advent of the Univalent Foundations Program/ Homotopy Type Theory and commands ever larger throngs of adherents. 
We will be primarily interested in a non-homotopic variant which we'll christen Arithmetic Type Theory for convenience. 

\begin{rmk}
Type theory is natively constructive, meaning in particular that the Principle of Excluded Middle does not hold in general. It is commonly supposed that constructive mathematics imposes too heavy a constraint on the tools a mathematician may use. \\
For many it comes as a surprise to learn that constructive mathematics is \emph{more general} than classical mathematics. Classical mathematical systems may in fact be embedded inside their constructive counterparts. Some favoured theorems will not hold as stated, but experience suggests that by changing the definitions slightly the spirit if not the letter of the law may be preserved. Rather than a annoying inconvenience searching for the alternative formulations often turns up new mathematics in well-trodden fields.
For a very clear exposition of the advantages of constructive mathematics and the relation with the internal language of topoi [of which arithmetic universes are a generalisation] see the first two chapters of \cite{Blechschmidt2017}.
\end{rmk}
In traditional foundations one learns that everything is a set. Set theory starts as a theory with one binary connective $\in$ and then proceeds by postulating a number of axioms with the hope of explaining all the possible behavior a set must have. Type theory starts from an entirely different perspective. It involves several kinds of declarations, called \emph{judgments}, which declare that something is a type or term of a type. These judgments are derived from rules which explain how new judgments can be made from old ones. As an example, suppose that we interested in the disjoint sum. Assuming that the type theory has the disjoint sum type, then the disjoint sum comes with the following rules:
\begin{itemize}
    \item If $A$ and $B$ are types, then $A+B$ is a type.
    \item If $a:A$ is a term of type $A$, then $\iota_1(a):A+B$ is a term of type $A+B$. Similarly, if $b:B$ is a term of type $B$, then $\iota_2(b):A+B$ is a term of type $A+B$.
\end{itemize}
Actually, more rules are needed which explain how the terms of a disjoint sum are used, see below. It is hoped that the general syntax is nevertheless elucidated. 

The type theory we need is a type theory constructed by Maietti \cite{maietti05b} in the style of Martin-L\"of \cite{martin-lof84}. One of the features of this kind of type theory is that the identity can be considered as a type. This gives us two different kinds of ways to talk about identities in the type theory. There are identities which are postulated by valid judgments $a=a':A$ that say that the terms $a$ and $a'$ are judged equal. So for the above example, we would need a rule that says that if $a=a':A$ is a valid judgment, then $\iota_1(a)=\iota_1(a'):A+B$ is a valid judgment. This kind of equality is understandably called \emph{judgmental equality}. But then there are \emph{propositional equalities}: if $A$ is a type and $a:A$ and $a':A$ are terms of that type then there is a type $a=_Aa'$. The idea is that if this type has a term, then the equality must be equal. So terms should be thought of as proofs, hence the name propositional equality. The rules of this equality work (approximately, the more formal definition follows later) as follows:
\begin{itemize}
    \item If $a=a':A$ is a valid judgment, then there is a term (proof) $\ast:a=_Aa'$.
    \item If $t:a=_Aa'$ is a term of the type, then $t=\ast:(a=_Aa')$. 'All proofs are equal'.
    \item If $t:a=_Aa'$ is a term of the type, then $a=a':A$ is a judgmental equality.
\end{itemize}
\begin{remark}

The last two rules makes our type theory an \emph{extensional type theory}. An \emph{intentional type theory} is one where one does not try to collapse the different inhabitans of the identity type. The distinction was first made by Martin-L\"of, who published his intentional type theory in 1975 \cite{martin-lof75} and his extensional one in 1984 \cite{martin-lof84}. The intentional type theory has the advantage that propositional equality is decidable. In extensional type theory the propositional equality is not decidable, but at first sight it seems more suitable for ordinary mathematics. However, it has been argued at length \cite{hofmann95} that intentional type theory with extensional features is actually the right way to do type theory.  Homotopy type theory is based on Martin-L\"of's intentional type theory, adding the univalence axiom, which says roughly that isomorphic structures may be identified. The main idea of homotopy type theory is that a type should not be seen as a type, but as a space (in the sense of homotopy theory). Homotopy type theory provides a new foundation of mathematics called \emph{univalent foundations}. This not only provides an easier way to build a proof-checkers of ordinary mathematics, it also provides a way to do interesting mathematics which has not come up under classical set theory: synthetic homotopy theory. See the Homotopy type theory book \cite{HoTT2013} for more information. 
\end{remark}

Maietti's type theory is an extensional type theory based on the extensional version of Martin-L\"of's type theory \cite{martin-lof84}.

Type theory reasons about declarations called \emph{judgments}. There are the judgments that say a particular structure is a context, type or term.
\begin{itemize}
    \item $\Gamma$ is a context (formally `$\Gamma$ cont').
    \item $A$ is a type in context $\Gamma$ (formally `$x:\Gamma \vdash A(x) \; \text{Type}$).
    \item $a:A$ is a term in context $\Gamma$ (formally $x:\Gamma \vdash a:A(x)$).
\end{itemize}
Then there are also the judgments which declare certain contexts, types or terms to be equal:
\begin{itemize}
    \item $\Gamma$ and $\Gamma'$ are equal as contexts (formally `$\Gamma = \Gamma'$ cont').
    \item $A$ and $B$ are equal types in context $\Gamma$ (formally $x:\Gamma \vdash A=B $).
    \item $a:A$ and $a':A$ are equal terms in context $\Gamma$ (formally $\Gamma \vdash a = a':A$).
\end{itemize}
These judgments are derived by certain rules allowing one to make a new valid judgment $\tau$ from a list of valid judgments $\sigma_1, \dots, \sigma_n$. Usually in type theory, this is written down in a strictly formalized way as
\[
\frac{\sigma_1 \dots \sigma_n}{\tau}
\]
We are going to be a bit more informal; we think that it is less daunting for an ordinary mathematician to see the statement `if $\sigma_1, \dots, \sigma_n$ are valid judgments, then $\tau$ is a valid judgment'. For the more formal description of Maietti's type theory see \cite{maietti05b}.

The first rules we take care of are the \emph{context rules}. A context is a list of variables in types $x_1 \in A_1, \dots, x_n:A_n$ where each type may depend on the previous one, so the contexts are generated by the rules
\begin{itemize}
    \item The empty list $\emptyset$ is a context.
    \item If $\Gamma$ is a context and $\Gamma \vdash A(x) \; \text{Type}$ a type dependent on $\Gamma$, then $\Gamma, x:A$ is a context, where $x$ is required to be a new variable not in $\Gamma$.
\end{itemize}
There are more \emph{structural rules}. A variable in a context can be declared a term:
\begin{itemize}
    \item If $\Gamma, x:A, \Delta$ is a context, then $\Gamma, x:A, \Delta \vdash x:A$ is a term in context.
\end{itemize}
There are some coherence rules for equality:
\begin{itemize}
    \item If $\Gamma \vdash A=B$ are equal types and $\Gamma \vdash a:A $ is a term of $A$ in context $\Gamma$, then $\Gamma \vdash a:B$ is a term of type $B$ in context $\Gamma$.
    \item If $\Gamma \vdash a=a':A$ are equal terms of type $A$, and $\Gamma \vdash A=B$ type are equal types, then $a=a':B$ are equal terms of type $B$.
\end{itemize}
There are some ordinary equality rules as well that say that equality behaves as an equivalence relation. We only write down these rules for types:  
\begin{itemize}
    \item If $\Gamma \vdash A \; \text{Type}$ is a type in context, then $A$ is equal to itself: $\Gamma\vdash A=A$.
    \item If $\Gamma \vdash A=B$ is an equality of types, then $\Gamma \vdash B = A$ is an equality of types as well.
    \item If $\Gamma \vdash A=B$ and $\Gamma \vdash B = C$ are equalities of types, then so is $\Gamma \vdash A = C$.
\end{itemize}
These rules must of course also hold for the equalities of contexts and terms.

The following two rules are derivable by induction, but they are nevertheless important. The \emph{substitution rules} are
\begin{itemize}
    \item If $\Gamma \vdash a:A$ and $\Gamma, x:A, \Delta \vdash b(x):B(x)$ are terms, then $\Gamma,\Delta[a/x]\vdash b(a):B(a)$, where $\Delta[a/x]$ is $\Delta$ with all mentions of $x$ replaced by $a$.
    \item If $\Gamma\vdash a:A$ and $\Gamma, x:A, \Delta \vdash b(x) = b'(x):B(x)$ then $\Gamma,\Delta \vdash b(a) = b'(a):B(a)$.
\end{itemize}
The \emph{rules of weakening} are
\begin{itemize}
    \item If $\Gamma\vdash A \; \text{Type}$ is a type and $\Gamma,\Delta \vdash b:B$ a term, then $\Gamma, \Delta \vdash b:B$ is a term.
    \item If $\Gamma\vdash A \; \text{Type}$ is a type and $\Gamma,\Delta \vdash b=b':B$ an equality of terms, then $\Gamma, x:A, \Delta \vdash b=b':B$ is an equality of terms.
\end{itemize}

\subsubsection*{Type constructors}

The remaining kind of rules that need to be discussed are the rules that belong to type constructors - a type constructor specifies how to build a new type from existing ones and how to use it. Each type constructor comes with five kind of rules:
\begin{enumerate}
    \item \emph{Formation rules} that say when the type constructor yields a new type.
    \item \emph{Introduction rules} that say how to define new terms of the newly constructed type.
    \item \emph{Elimination rules} that say how terms of the type are used.
    \item \emph{Computation rules} that say how the elimination and introduction rules are combined.
    \item \emph{Uniqueness rules} that say how the terms of the type are uniquely determined by the elimination rules. They are often omitted, for they are usually derivable.
\end{enumerate}
Then each time a new type or term is introduced there need to be new rules that state that the equality is well-behaved. For instance, when the disjoint sum type constructor and its terms are defined there are additional equality rules
\begin{itemize}
    \item If $\Gamma \vdash A = A'$ and $\Gamma \vdash B=B'$ are type equalities, then there is the type equality $\Gamma \vdash A+B=A'+B'$. 
    \item If $\Gamma \vdash a=a':A$ are equal terms, then $\Gamma \vdash \iota_1(a) = \iota_1(a'):A+B$ are also equal terms.
\end{itemize}
These usually remain unstated, for no other reason than to save space. The rules for the term and type constructors are given by Maietti as follows
\begin{figure}[htp]
    \centering
    \includegraphics[width=14cm]{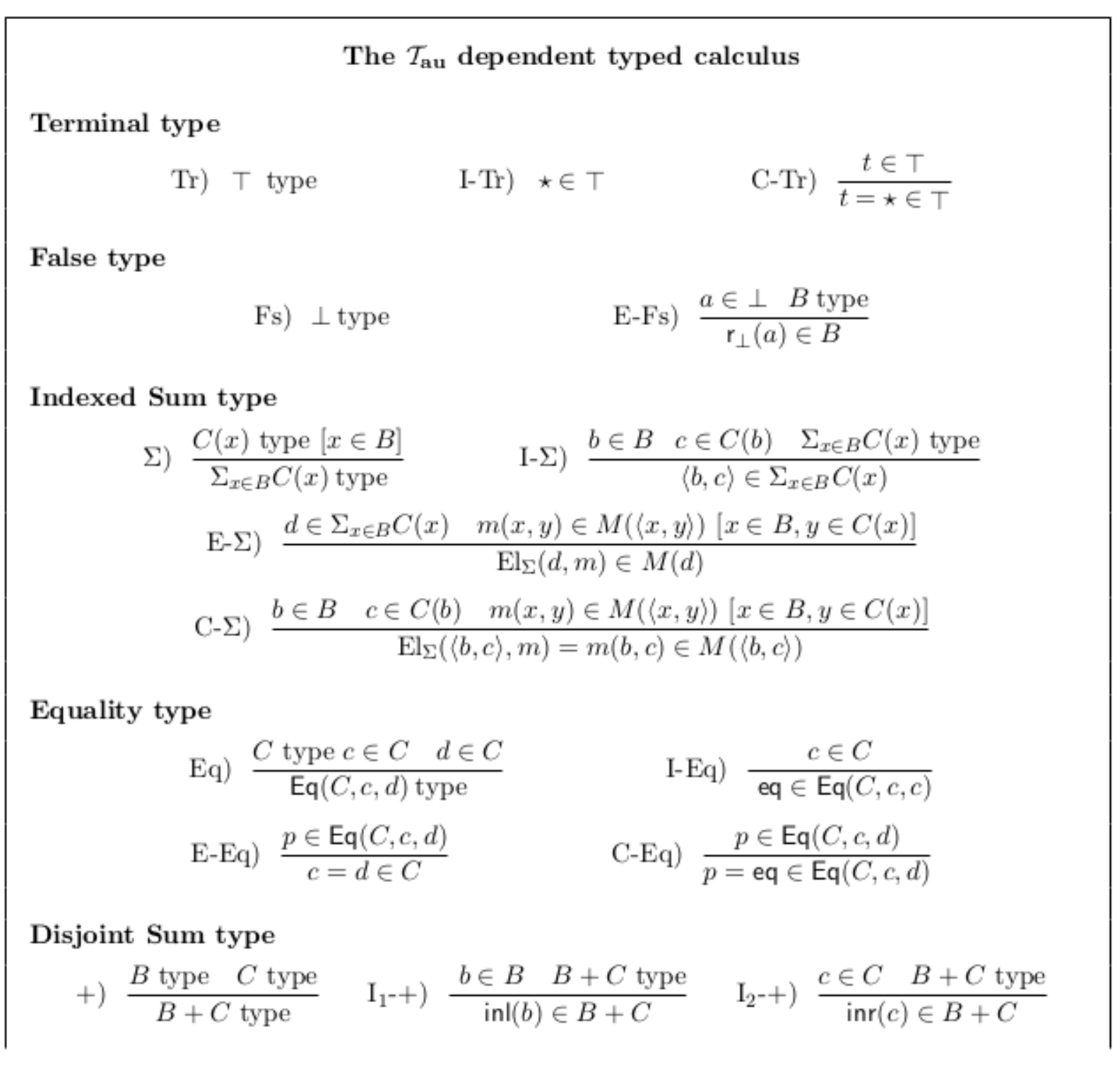}
\end{figure}

\begin{figure}[htp]
    \centering
    \includegraphics[width=14cm]{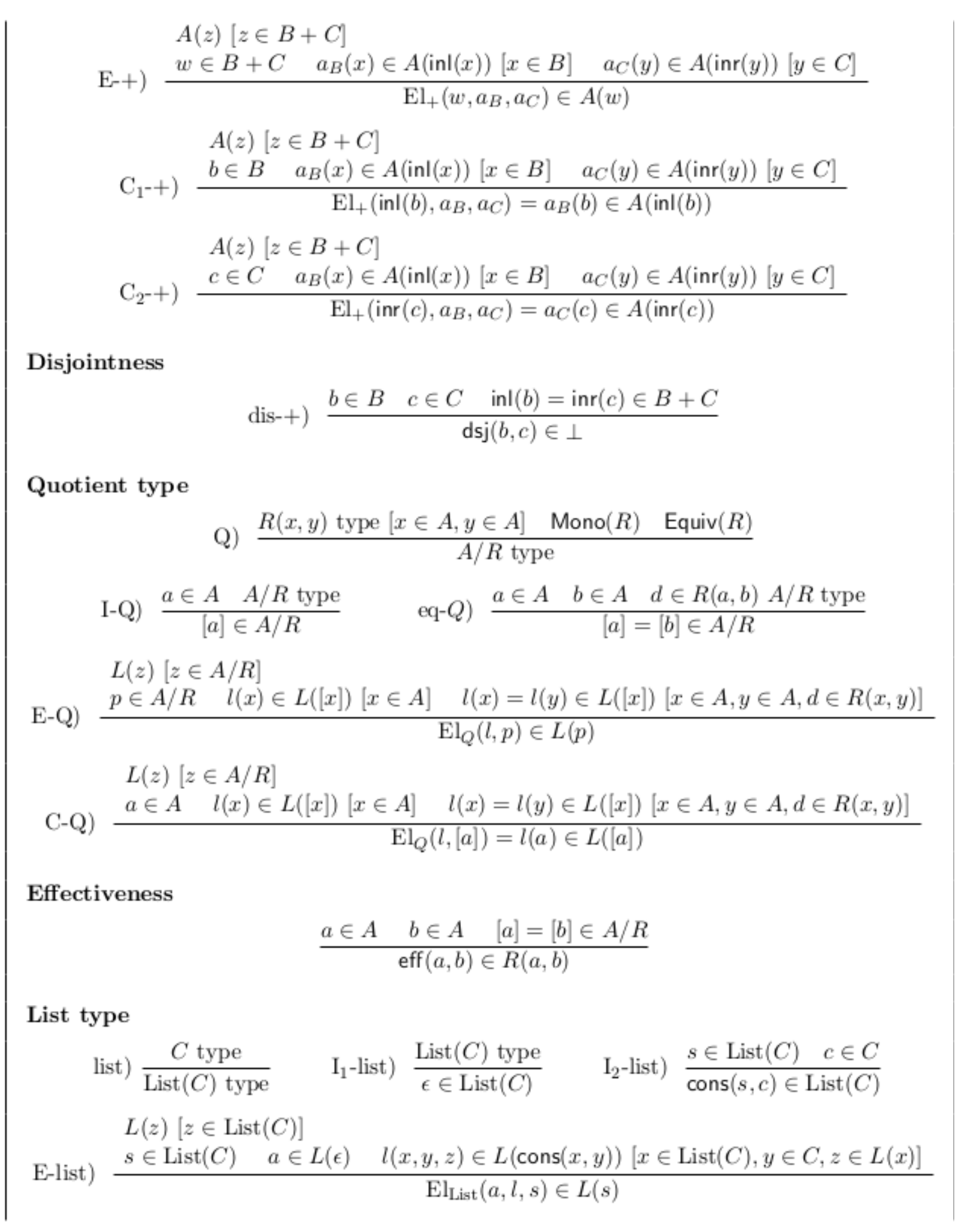}
  \vspace*{\floatsep}
    \includegraphics[width=14cm]{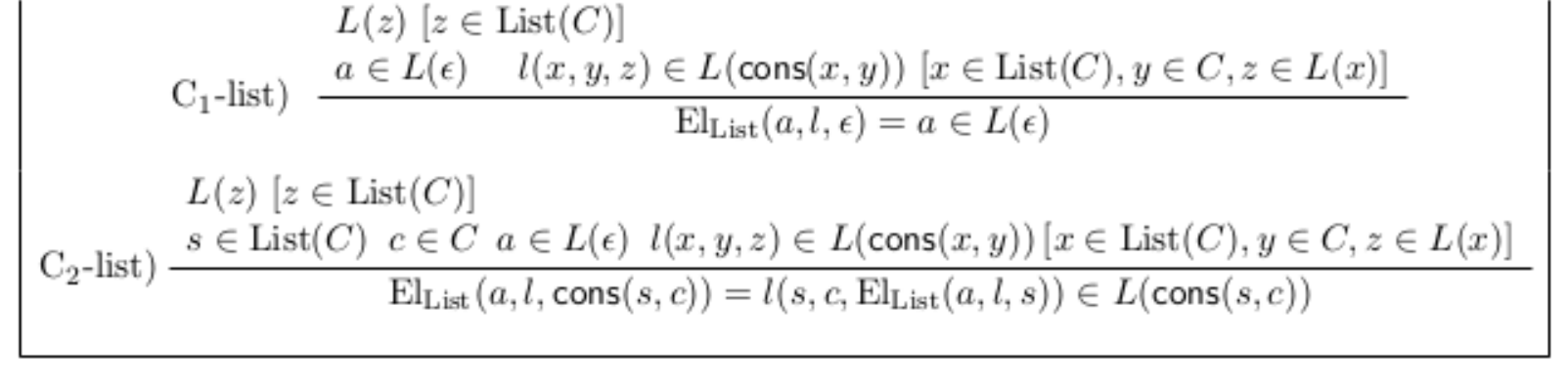}
    \caption{Reproduced from \cite{maietti03}}
\end{figure}

\pagebreak

The basic type theory for arithmetic universes $\mathcal{A}u$ is the dependent type theory that consists of all the above type constructors. It is possible to extend the basic type theory $\mathcal{A}u$ to something larger.

\begin{defn}
A \emph{theory of $\mathcal{A}u$} is a typed calculus $\mathcal{T}$ that consists of the type constructors of $\mathcal{A}u$  containing possible additional type judgments $\Gamma \vdash A \; \text{Type}$, term judgments $\Gamma \vdash a:A$, type equality judgments $\Gamma \vdash A=B$ and term equality judgments $\Gamma \vdash a=_Aa'$. There are some obvious restrictions: $\Gamma \vdash a:A$ can only be added if $\Gamma \vdash A \; \text{Type}$ is derivable in $\mathcal{T}$, $\Gamma \vdash A=B$ if $\Gamma\vdash A \; \text{Type}$ and $\Gamma \vdash B \; \text{Type}$ are derivable  in $\mathcal{T}$, and $\Gamma \vdash a=_Aa'$ if $\Gamma \vdash a:A$ and $\Gamma \vdash a':A$ are derivable in $\mathcal{T}$.
A morphism of theories $\mathcal{T} \to \mathcal{T'} $ is an assignment of judgements of $\mathcal{T}$ to the judgements of $\mathcal{T}'$ which preserve the rules. 
\end{defn}

\begin{rmk}
The internal language of an arithmetic universe has been mentioned, but what does this mean exactly? Roughly, it means the following: given an arithmetic universe $\mathcal{U}$ we may interpret types in $\mathcal{T}_{AU}$ as objects in $\mathcal{U}$ and types in context in its slice categories $\mathcal{U}/c$ (where $c$ is the context). Operations like sum and product of types in the type theory correspond to sum and product in the category $\mathcal{U}$.
An operation like the existence operator $\exists$ is a little more involved, using the image factorisation that an arithmetic universe possesses by virtue of it being a 'regular' category. 
\end{rmk}

\begin{ex}
Let $\mathcal{C}$ be any category. Let $\mathcal{T}_{\mathcal{C}}$ be the theory  which is obtained by adding to the calculus of $\mathcal{A}u$ one closed type (i.e. in which the context is empty) $\vdash A \; \text{Type}$ for every object $A$ of $\mathcal{C}$, a term $x:A \vdash f(x):B$ for every morphism $f:A \to B$ in $\mathcal{C}$, a type equality $A=B$ type if $A$ and $B$ are the same objects in $\mathcal{C}$ and a term equality $x:A \vdash f(x)=g(x):B$ if $f$ and $g$ are the same morphisms in $\mathcal{C}$. 
\end{ex}
\begin{rmk}
One should be careful to distinguish Type Theory in general, and specific type theories.
In that vein, be aware that there is a difference between $\mathcal{T}_{AU}$ and $T_{\mathcal{U}}$. Both are often called the 'internal language'. Roughly, the latter regards the arithmetic universe $\mathcal{U}$ as a distinct mathematical universe in which one may do mathematics, while the former concerns constructions that makes sense for any Arithmetic Universe. On the other hand they are related: $T_{\mathcal{U}}$ is build from $\mathcal{T}_{AU}$ by adding certain types, terms, and equalities --- but no new rules!
In short, $T_{\mathcal{U}}$ concerns mathematics internal to a given mathematical universe $\mathcal{U}$ while $\mathcal{T}_{AU}$ is the language that is used between different mathematical universes.
\end{rmk}

\begin{defn}
Let $\mathcal{T}$ be a theory in the above sense. The \emph{syntactic category of $\mathcal{T}$} is a category $\mathcal{C}_{\mathcal{T}}$ with as objects the closed types $\vdash A \; \text{Type}$ of $\mathcal{T}$. If $\vdash A \; \text{Type}$ and $\vdash B \; \text{Type}$ are two closed types, then a morphism from $\vdash A \; \text{Type}$ to $\vdash B \; \text{Type}$ is defined as a term $x:A \vdash f(x):B$, where two morphisms $x:A \vdash f(x):B$ and $x:A \vdash g(x):B$ are considered to be equal if $x:A \vdash f(x)=g(x):B$ is derivable. Composition is given by substitution and the identity is represented by $x:A \vdash x:A$. 
\end{defn}
Let $\mathcal{T}$ be a theory of $\mathcal{A}u$. Then $\mathcal{C}_{\mathcal{T}}$ is an arithmetic universe.

\begin{rmk}[Syntax-Semantics adjunction] We will focus on the initial object $\mathcal{U}_{in}$ of the category of Arithmetic Universes $\mathcal{AU}$. Initiality of the model will mean roughly that $\mathcal{T}_{AU}=\mathcal{T}_{\mathcal{U}_{in}}$; one also says that $\mathcal{U}_{in}$ is the \emph{syntactic category} for the type theory $\mathcal{T}_{AU}$.  
Let $\mathcal{AU}$ denote the category of arithmetic universes, and morphisms preserving all structure, let $\mathcal{T}_{AU}-Type$ denote type theories over $\mathcal{T}_{AU}$ with maps interpretations preserving structure up to isomorphism --- we will explain this in more detail later. We have an 'syntax-semantics' adjunction
\[
\begin{tikzcd}
\mathcal{T}_{AU}-Type\ar[r,bend left,"Syn",""{name=A, below}] & \mathcal{AU}\ar[l,bend left,"Lan",""{name=B,above}] \ar[from=A, to=B, symbol=\dashv]
\end{tikzcd}
\]
where the $Lan$-functor produces the internal language $T_{\mathcal{U}}$ of an arithmetic universe $\mathcal{U}$ and the $Syn$-functor the syntactic category $\mathcal{C}_\mathcal{T}$ for a given $\mathcal{T}_{AU}$-type theory $T$. In this case, the adjunction is in fact an equivalence.
\end{rmk}
Given a theory $\mathcal{T}$ we may construct its syntactic category $Syn(\mathcal{T})$. In our case we will take $\mathcal{T}=\mathcal{T}_{AU}$. This type theory can encode primitive recursive arithmetic.

\begin{thm}
The syntactic category coincides with Joyal's construction: $Syn(\mathcal{T}_{AU})\cong \mathcal{U}_0 := Pred(\Sigma_0)_{ex/reg}$.
\end{thm}
\begin{proof}
Theorem 6.6 of \cite{maietti03}.
\end{proof}
Given a sentence $\phi$ in the language of $\mathcal{T}$ we may ask whether $\mathcal{T}$ proves $\phi$. The sentence $\phi$ also appears as a subobject of $1$ in $Syn(\mathcal{T})$. Its provability is exactly the assertion that it is the maximal subobject of $1$. This motivates the following definition. 
\begin{defn}
An arithmetic universe $U$ is \emph{consistent} if $0\in \mathcal{P}1$ and $1 \in \mathcal{P}1$ are distinct. A theory $U$ is \emph{complete} if given $u\in \mathcal{P}1$, we have $u=0$ or $u=1$ [in the meta-theory].  
\end{defn}
\begin{rmk}
A peculiarity of the above definition is that syntactic completeness is a sort of Boolean property of the syntactic category $Syn(\mathcal{T})$. Indeed, this is already apparent in the classical formulation of syntactic completeness as the existence of a proof $\phi$ or its negation $\neg \phi$ for all sentences $\phi$ in the theory. The above formulation brings intuitionistic aspects of the provability predicate in direct contact with the constructive nature of the internal languages of categories. 
\end{rmk}
\begin{remark}
There are traditionally \emph{two} notions of both completeness and consistency, a semantic one and a syntactic one. 
A theory $\mathcal{T}$ is semantically consistent \footnote{This is sometimes also called satisfiability of $\mathcal{T}$.} if it has a model $M$. A theory is semantically complete\footnote{This is sometimes called validity of $T$.} if for any formula $\phi$ in the language of $\mathcal{T}$ it is provable if and only if it is true in all models.
A theory $T$ is syntactically consistent if it does not prove a contradiction, i.e. if it does not derive the falsum $\bot$. A theory is syntactically complete if for any $\phi$ either $\mathcal{T}$ proves $\phi$ or it proves its negation $\neg \phi$.
The semantic side and the syntactic side are often conflated, but it is important to keep them distinct. The semantic side always refers to a class of models $S=\{M_i\}_{i\in I}$ for the theory $T$, while the syntactic theory only refers to the theory. 
The Godel Completeness theorem states that classical first order theories are semantically complete with respect to Set-models. Logics that are not classical first-order often fails to have enough Set-based models.  In categorical logic one works instead with category-based models. Despite being of indisputable interest, semantic notions of completeness and consistency will not occupy us here.

The first G\"odel Incompleteness Theorem states that a recursively enumerable theory $\mathcal{T}$ that is consistent and can encode arithmetic is syntactically incomplete. 
The Second Incompleteness Theorem states if a recursively enumerable theory $\mathcal{T}$ that can encode a weak fragment of arithmetic \emph{proves its own consistency} then in fact it is \emph{not consistent}. 
\end{remark}

\section{The G\"odel Incompleteness Theorems}
In this section we will give a proof of G\"odel's second Incompleteness Theorem using arithmetic universes.

\begin{defn} Let $P$ be an object of a category $C$. We say $P$ is \emph{projective} if given any epimorphism $S \twoheadrightarrow T$ we have a lift
\[
\begin{tikzcd}
&S \arrow[d]\\
P \arrow[ur, dotted] \arrow[r, twoheadrightarrow]& T 
\end{tikzcd}
\]
\end{defn}
At this point a translation of  Cantor's Diagonal argument into categorical terms is given for motivation and later comparison. 
\begin{thm}[Cantor]
Let $\mathcal{E}$ a topos in which $1$ is projective. If there exists an enumeration $f:A \twoheadrightarrow \mathcal{P} A $ , then  $\mathcal{E}$ is degenerate.
\end{thm}
\begin{proof}
 Indeed, suppose $f$ exists, form the pullback 
\[
\begin{tikzcd}
D \arrow[d,hook] \arrow[dr, phantom, "\fiberproduct", very near start] \arrow[rrr]&&& 1 \arrow[d,hook,"false"]\\
A \arrow[r,"\Delta_A"]&A \times A \arrow[r,"1_A \times f"] & A \times \mathcal{P}A \arrow[r,"eval"] &\mathcal{P}1
\end{tikzcd}
\]
and consider $name{D}: 1 \to \mathcal{P}A$. By definition of projectivity of $1$ there is a lifting $a$ of $\name{D}$ 
\[
\begin{tikzcd}
&A \arrow[d]\\
1 \arrow[ur, dotted, "a"] \arrow[r, "nameD"]& \mathcal{P}A 
\end{tikzcd}
\]
then $\ext{a\in D}$ is defined by the pullback 
\[
\begin{tikzcd}
\ext{ a \in D} \arrow[r,hook]\arrow[d,hook]  \arrow[dr, phantom, "\fiberproduct", very near start]  & D \arrow[d, hook]\\
1 \arrow[r, "a"]  &A
\end{tikzcd}
\]
the composite pullback 
\[
\begin{tikzcd}
\ext{a \in D} \arrow[r,hook]  \arrow[d,hook] \arrow[dr, phantom, "\fiberproduct", very near start] & D\arrow[d,hook]  \arrow[dr, phantom, "\fiberproduct", very near start]  \arrow[rrr] &&& 1 \arrow[d,hook,"false"]\\
1 \arrow[r, "a"]& A \arrow[r,"\Delta_A"]& A\times A \arrow[r,"1_A\times f"]& A \times \mathcal{P} \arrow[r,"eval"]& \mathcal{P}1 
\end{tikzcd}
\]
is then just 
%\begin{align}\label{cantor2}
\[
\begin{tikzcd}
\ext{ a \in D} \arrow[dr, phantom, "\fiberproduct", very near start] \arrow[d,hook]\arrow[r,hook]&1 \arrow[d,hook,"false"]\\
1 \arrow[r, "\name{\ext{ a\in D}}"]& \mathcal{P}1.
\end{tikzcd}
\]
%\end{align}
It follows that 
\[
\begin{tikzcd}
\ext{ a \in D} \arrow[r,hook] &1 \arrow[r, shift left=1.5ex,"\name{0}"] \arrow[r,swap, "\name{1}"]& \mathcal{P}1
\end{tikzcd}
\]
is an equalizer. Indeed, suppose 
\[
\begin{tikzcd}
Z \arrow[r,"!_Z"] &1 \arrow[r, shift left=1.5ex,"\name{0}"] \arrow[r,swap, "\name{1}"]& \mathcal{P}1
\end{tikzcd}
\]
is a commutative diagram.  Then we have, since $Sub_{\mathcal{E}}(Z) = \hom_{\mathcal{E}}(Z, P1)$ is a poset with $\name{0}\ \circ\ !_Z$ as bottom and $\name{1}\ \circ\ !_Z$ as top, that $\name{0}\ \circ\ !_Z \leq \ext{ a \in D} \leq \name{1}\ \circ\ !_Z = \name{0}\ \circ\ !_Z$. Thus $\ext{ a \in D} = \name{1}\ \circ\ !_Z = \name{0}\ \circ\ !_Z$. Then, by one of the two pullbacks above, $!_Z$ factors through $\ext{ a \in D}$. 
It then follows that $ (\ext{a \in D} \hookrightarrow 1) \simeq (0\hookrightarrow 1) $.

Next, $\ext{ a \in D}$ is also the pullback of $false$ along $false$, so $(\ext{ a \in D} \hookrightarrow 1 \simeq (1 \hookrightarrow 1)$ as well; but then $0=1$ .
\end{proof}

\begin{defn} Let $U$ be an arithmetic universe. An \emph{AU-object} $\mathcal{E}$ is an internal category 
\[
\mathcal{E}_1\times_{\mathcal{E}_0}\mathcal{E}_1 \xrightarrow{\circ} \mathcal{E}_1 \rightrightarrows \mathcal{E}_0
\]
such that internally $\mathcal{E}$ is a list-arithmetic pretopos. 
Arithmetic universe objects (AUO's) and their internal functors form a (large, external) category $AUO_{\mathcal{U}}$; hence it makes sense to talk about limits, colimits, initiality, et cetera of AU-objects. 
\end{defn}
\begin{defn}
A sketch is quadruple $K=(G, U, D,C)$ where $G$ is a graph, $U: G_0 \to G_1$ is a function, $D$ is a collection of diagrams in $G$ and $C$ is a collection of cones in $G$.
A sketch morphism $T: K \to K'$ is a graph homomorphism $T: G \to G'$
such that $(i) T \circ U= U'\circ T$, (ii) every diagram in $D$ is mapped to a diagram in $D'$ and (iii) every cone in $C$ is mapped to a cone in $C'$.
\end{defn}
\begin{defn}
If $\mathcal{C}$ is a category then the underlying sketch $K_{\mathcal{C}}=(G,U,D,C)$ is given as
\begin{itemize}
    \item $G$ is the underlying graph of $\mathcal{C}$
    \item $U$ is the map which picks out the identity arrows of $\mathcal{C}$
    \item  $D$ is the collection of all commutative diagrams of $\mathcal{C}$
    \item $C$ is the collection of all limit cones of $\mathcal{C}$.
\end{itemize}
\end{defn}
\begin{defn}
 A model for a sketch $K$ in a category $\mathcal{C}$ is a morphisms of sketches from $K$ to the underlying sketch $K_{\mathcal{C}}$ of $\mathcal{C}$.
\end{defn}
Note that the models $Mod(K,\mathcal{C})$ of $K$ in $\mathcal{C}$ form a category.
\begin{lem} There is a sketch $K_{Skolem}$ of Skolem categories.
\end{lem}
\begin{proof}
See Lemma 7.12 in \cite{morrison96}.
\end{proof}
\begin{thm}
Let $\mathcal{E}$ be any arithmetic universe. Internally, we may construct the initial arithmetic universe object $\mathbb{U}_0(E)$.
\end{thm}
This is Theorem 7.13 of \cite{morrison96}. We give a sketch of the proof. 
The nontrivial part is the construction of the internal initial Skolem theory $\Sigma_0'$. This relies fundamentally on the fact that Arithmetic Universes have parameterized list object and hence may implement primitive recursion. 
The reader is warned that the second part of the proof is not terribly enlightening, but may nevertheless give a sense of "what's involved".
\begin{proof} 
We consider the sketch $K_{Skolem}$ and the empty sketch $K_0$. We have a morphism of sketches $K_0 \to K_{Skolem}$ which induces the forgetful functor $K: K_{Skolem}(\mathcal{U}_0) \to K_0(\mathcal{U}_0)$; by the free model theorem [Theorem 29 of \cite{palmgrenvickers2007}] there is a free left adjoint $L: \{\bullet\}=K_0(\mathcal{U}_0)\to K_{Skolem}(\mathcal{U}_0)$. The internal initial Skolem theory is $\Sigma':= L(\bullet)$.

Let the following denote $\Sigma'$
\[
\begin{tikzcd}
\Sigma'_1 \arrow[r,  shift right=1.5ex, "\delta_1"] \arrow[r,shift left=1ex, "\delta_0"] & \Sigma'_0 \arrow[r,"e"] & \Sigma'_1 
\end{tikzcd}
\begin{tikzcd}
\Sigma'_1 \times_{\Sigma'_0} \Sigma_1' \arrow[r, "m"] & \Sigma'_1
\end{tikzcd}
\]
We construct internally $Pred(\Sigma')$. The object of objects of $Pred(\Sigma')$ is all predicates, i.e. the following equaliser
\[
\begin{tikzcd}
&& \Sigma_1'\times_{\Sigma_0'} \Sigma_1' \arrow[dr,"m"] &\\
Pred(\Sigma')_0  \arrow[r,hook]& \Sigma'_1 \arrow[ur,"\Delta"] \arrow[rr,"id"] && \Sigma'_1 
\end{tikzcd}
\]
The arrows from $A$ to $B$ in $Pred(\Sigma')$ are equivalence classes of the set 
\[
\{f: \mathbb{N}_{\mathcal{U}_0}\to \mathbb{N}_{\mathcal{U}_0}| A \leq B \circ f\}
\]
which is internally constructed as the equalizer
\[
\begin{tikzcd}
&& \Sigma_1' \times Pred(\Sigma'_0) \arrow[dr, "\leq"]&\\
X\arrow[r, hook] &  \Sigma_1'\times Pred(\Sigma_0')\times Pred(\Sigma_0') \arrow[ur, "(\pi_2{,} m \circ (\pi_1 {,}\pi_3))"] \arrow[rr,"True"]&&\Sigma'_0  
\end{tikzcd}
\]
Next, define 
\[d_0 : X \hookrightarrow  \Sigma_1'\times Pred(\Sigma_0')\times Pred(\Sigma_0') \xrightarrow{\pi_2} Pred(\Sigma_0')  \]
\[
d_1 : X \hookrightarrow  \Sigma_1'\times Pred(\Sigma_0')\times Pred(\Sigma_0') \xrightarrow{\pi_3} Pred(\Sigma_0')  
\]
We have $f\sim g$ if and only if $A \leq eq(f,g)$. We construct the pullback of $f,g$ such that $f,g$ have the same source and target as
\[
\begin{tikzcd}
Y \arrow[d,"p_1"] \arrow[r, "p_2"] & X \arrow[d, "(d_0{,}d_1)"]\\
X  \arrow[r, "(d_0{,}d_1)"] & Pred(\Sigma_1') \times Pred(\Sigma_1')
\end{tikzcd}
\]
The subset $R$ of $Y$ with $(f,g) \in R$ if and only if $src(f)\leq eq(f,g)$ is build as the equalizer:
\[
\begin{tikzcd}
&& Pred(\Sigma_0') \times Y \times Y  \arrow[r]& Pred(\Sigma_0') \times \Sigma_0' \arrow[dr] & \\
R \arrow[r,hook] & Y \arrow[ur,"(src\circ p_1{,}p_1{,}p_2)"]  \arrow[rrr, "True"]&&& \Sigma_0'
\end{tikzcd}
\]
The object of morphisms $Pred(\Sigma')_1$ is given as the quotient of $R \hookrightarrow Y$.
The exact-regular completion is a similar mess. 
\end{proof}
\begin{remark}
It has become clear that, although quotidian, the need for redoing external constructions internally is a burdensome process. Ideally, one would have access to a device that could make these internal workings completely routine. The internal language provides such a device.
\end{remark}

\begin{defn}
 Given an internal category $\mathbb{C}$ inside a category $D$ with finite limits we may take the externalisation, taking objectwise global sections: $Ext(\mathbb{C})=(Hom_D(1,Ob\mathbb(C)),Hom_D(1,Mor(\mathbb{C})))$ with the obvious maps. We obtain an external category $Ext(\mathbb{C})$.
\end{defn}
\begin{remark}
There is a different kind of externalization that is more common. That is the Grothendieck-externalization $\int \mathbb{C}$ whose objects are $X:I  \to Ob(\mathbb{C})$ and morphisms are diagrams
\[
\begin{tikzcd}
&Mor(\mathbb{C}) \arrow[d, " <cod{,}dom>"]\\
I \arrow[ur, "f"] \arrow[r, " <X{,} Y>"]& Ob(\mathbb{C})
\end{tikzcd}
\]
It is well-known that this gives a 2-functor $\int:\mathbb{C}(U_0) \to Fib/U_0$ from the 2-category of internal categories to fibrations over $U_0$. This functor is in fact fullly faithful, see lemma 2.3.3 of B3 in \cite{Elephant}. This 2-functor sends internal (co)limits to fibered (co)limits.
The simple externalization $Ext(\bullet)$ is simply the fiber over the terminal object, in other words $Ext(\mathbb{C})=\int \mathbb{C}$. That means that if $\mathbb{C}$ has a certain limit or colimit $Ext(\mathbb{C})$ has that (co)limit, and it is moreover stable under pullback. 
\end{remark}
\begin{remark}
There is an important difference between local existence or global existence. Take for example the statement that an internal category $\mathbb{C}$ has an internal terminal object.
Local existence would say that in the Kripke-Joyal semantics $\mathbb{C}\models \exists \, 1:\mathbb{C} \, \textit{such that 1 is a terminal object}$ which is different from the statement that there exists an object $1_{\mathbb{C}}:1_{U_0}\to  \mathbb{C}$ such that $\mathbb{C} \models 1_{\mathbb{C}} \textit{ is a terminal object}$. In the construction we actually get global existence for the objects that we construct and this is key. 
\end{remark}
\begin{prop} Let $E$ be an arithmetic universe equipped with an internal arithmetic universe object $\mathbb{U}$.
The externalisation $Ext(\mathbb{U})$ of an internal arithmetic universe object $\mathbb{U}$ is an (external) arithmetic universe. 
\end{prop}

\begin{proof}
An arithmetic universe is a list-arithmetic pretopos. That means that $\mathbb{U}_0$ is internal category that is internally
\begin{itemize}
    \item Finitely complete \\
By the above remarks on the Grothendieck-externalization $\int$ this is immediate. 
    \item Finite disjoint stable coproducts \\
Similarly. In fact, we have all finite colimits. 
    \item Parameterized list objects\\
We have internal list objects in $\mathbb{U}_0$. As we have global existence of our objects this means that for any object $A:1 \to Ob(\mathbb{U}_0)$ there is a diagram  $(1 \xrightarrow{cons} A \xleftarrow{app} A \times L(A)): \mathbb{U}_0(1)$ such that $\mathbb{U}_0 \models \name{\textit{For all}\, 1 \xrightarrow{c} Y \xleftarrow{f}Y \times A \: \: \exists ! rec(f,g): L(A) \to Y \, \textit{such that the natural diagrams commute}}$.
Let's write that out using the Kripke-Joyal semantics. The above maybe be translated as
\begin{gather*}
\textit{for all} \: \:g_1:I_1 \to 1 \, (B\xrightarrow{c} Y \xleftarrow{f} Y\times g_1^{\ast}A ) : \mathbb{U}_0(I_1) 
\:\:\textit{ there exists}\:\: p_2: I_2 \twoheadrightarrow I_1 \\ \textit{and a unique}\,\, (rec(c,f): p_2^{\ast} \to p_2^{\ast}p_1^{\ast}L(A)):\mathbb{U}_0\quad \textit{such that}
\end{gather*}
\[
I_2 \models
\begin{tikzcd}
p_2^{\ast}g_1^{\ast} 1 \arrow[dr]\arrow[r]& p_2^{\ast}g_1^{\ast} L(A) \arrow[d, "rec(c{,}f)" ]  & \arrow[l] p_2^{\ast}g_1^{\ast}L(A) \times p_2^{\ast}g_1^{\ast} A \arrow[d, " rec(c{,}f)"] \\
&p^2{\ast}Y & \arrow[l] p_2^{\ast}Y \times p_2^{\ast}g_1^{\ast} A
\end{tikzcd}
\]
We may specialise to $I_1=1$. Then we have 
\[
\begin{tikzcd}
I_2 \arrow[d, two heads] \arrow[rr,"rec(f{,}c)"]&& Mor(\mathbb{U}_0) \arrow[d]\\
1 \arrow[rr, "p_2^{\ast}Y\times p_2^{\ast}L(A)"] && Ob(\mathbb{U}_0)\times Ob(\mathbb{U}_0)
\end{tikzcd}
\]
By uniqueness of $rec(c,f)$ it descends to $1$ by effectiveness of quotients in an AU. So we see that internal list objects give us external parameterized list objects!
    \item Regular \\
    which is equivalently the statement that $\mathbb{U}_0$ is (i) finitely complete (ii) has coequalizers by kernel pairs and (iii) coequalizers by kernel pairs are preserved by pullback. The last condition does not involve any existential quantifiers hence specializing to $I_1=1$ yields the regularity of $\int \mathbb{U}_0(1)$.
    \item Exactness \\
    This means that for any given equivalence relation $R\hookrightarrow A \times A$ we have 
    \[
    \mathbb{U}_0\models \name{R\,\:\textit{ is a kernel pair}}
    \]
    i.e. if $A/R$ denotes the coequalizer of $\begin{tikzcd}
A\times A \arrow[r, " p_1", shift left=3] \arrow[r, "p_2", shift right=2] & A
\end{tikzcd}$ then we claim that $\mathbb{U}_0\models \name{R=A\times_{A/R}A}$ which is a simple equality, containing no existence quantifiers hence by the Kripke-Joyal semantics it follows that $Ext(\mathbb{U}_0)$ is exact.
    
\end{itemize}
\end{proof}
Let $U_0$ be the initial arithmetic universe, and $\mathbb{U}_0'$ its internal initial arithmetic universe. Let $\mathbb{N}$ denote the natural number object in $U_0$ and $\mathbb{N}'$ the natural number object in $\mathbb{U}_0'$.
Construct the externalization $U_0':= Ext_{U_0}(\mathbb{U}_0')$. 
Since $U_0$ is the initial arithmetic universe there is the initial functor $R: U_0 \to Ext(\mathbb{U}_0')$. Roughly speaking it interprets any construction that can be done in a general arithmetic universe in the specific universe $Ext(\mathbb{U}_0')$, which why we'll often denote $R(A)=A'$ for $A\in U_0$.

We also have a global section functor $\Gamma: Ext(\mathbb{U}_0)' \to U_0$. Let us see how it acts.
\textbf{On objects:} we construct the generic family of objects as a pullback
\[
\begin{tikzcd}
\mathbb{C}_1(1,\bullet) \arrow[d, " p"] \arrow[r] & \mathbb{C}_1 \arrow[d, " < cod{,}dom>"  ] \\
\mathbb{C}_0 \arrow[r, " < 1 {,} Id>" ] & \mathbb{C}_0 \times \mathbb{C}_0
\end{tikzcd}
\]
Given a representing code/name/arrow $X:1 \to \mathbb{C}_0$ for an object in $U_0'$ we compute $\Gamma(X)$ as the pullback
\[
\begin{tikzcd}
\Gamma(X) \arrow[d] \arrow[r] & \mathbb{C}_1(1',\bullet) \arrow[d, "p" ] \\
1\arrow[r] & \mathbb{C}_0 
\end{tikzcd}
\]
where $p: \mathbb{C}_1(1',\bullet) \to \mathbb{C}_0$  denotes the projection to the codomain. 

\textbf{On arrows:} let $T_0$ be the pullback
\[
\begin{tikzcd}
T_0 \arrow[r] \arrow[d] & \mathbb{C}_1(1',\bullet) \arrow[d,"p"] \\
 \mathbb{C}_1 \arrow[r, "dom" ] & \mathbb{C}_0
\end{tikzcd}
\]
let $T_1$ be the pullback
\[
\begin{tikzcd}
T_1 \arrow[r] \arrow[d, "r"] & \mathbb{C}_1(1',\bullet) \arrow[d,"p"]\\
\mathbb{C}_1 \arrow[r, "cod"] & \mathbb{C}_0 
\end{tikzcd}
\]
The way we constructed $T_0$ implies it is a subobject of $\mathbb{C}_2=\mathbb{C}_1 \times_{\mathbb{C}_0}\mathbb{C}_1$ and the composition map $\circ:\mathbb{C}_2 \to \mathbb{C}_1$ restrict to a map $T_0 \to \mathbb{C}_1(1,\bullet)$ such that
\[
\begin{tikzcd}
T_0 \arrow[r," \circ"] \arrow[d] & \mathbb{C}_1(1',\bullet) \arrow[d,"p"] \\
\mathbb{C}_1 \arrow[r, "cod" ] & \mathbb{C}_0
\end{tikzcd}
\]
commutes. By the universal property of the pullback we obtain a map $s: T_0 \to T_1$. The generic family is given by 
\[
\begin{tikzcd}
T_0 \arrow[rr," s" ] \arrow[dr] && T_1 \arrow[dl, " True' " ]\\
& \mathbb{C}_1 &
\end{tikzcd}
\]

Given a code for an arrow $1 \xrightarrow{f} \mathbb{C}$ with domain $R= dom \circ f$ and codomain $S=cod \circ f$ we pullback
\[
\begin{tikzcd}
f^{\ast}T_0 \arrow[dr] \arrow[rr,"f^{\ast}(s)"] && T_1 \arrow[dl] \\
&1&
\end{tikzcd}
\]
where we notice that $f^{\ast}(T_0)=R,f^{\ast}(T_1)=S$, so we get a map $\Gamma(f): R \to S$.
\begin{caution}
The functor $\Gamma$ is \emph{not} a morphism of arithmetic universes!
\end{caution}
\iffalse
\begin{defn}
An appeal to the glueing construction of Freyd is needed. Take a topos $\mathcal{E}$ and consider the global sections functor $\Gamma: \mathcal{E} \to Set $: form a  new category $Gl_{\Gamma}\mathcal{E}$ whose objects are maps $\alpha: S \to \Gamma(A)$ where $A \in \mathcal{E}, S \in Set$. The construction yields a new topos, alternatively called the Artin-Wraith gluing along $\Gamma$, the 'graph topos of $\Gamma$', the scone construction or the \emph{Freyd cover}. An exactly analogous construction goes through for arithmetic universes.
\end{defn}
\fi

\begin{lem} \label{FreydMap} Let $\eta_A:A \in U_0$. We have a map $A \to \Gamma(A')=\square A $.
\end{lem}
\begin{proof}
This follows from the existence of the Freyd cover obtained from gluing along the global section functor $\Gamma$ and the initiality of $U_0$. We consider the arithmetic universe $U_0'=Ext(\mathbb{U})$. There is a functor $\Gamma: Ext(\mathbb{U}) \to \mathbb{U}_0$. 
Construct the Artin-Wraith gluing or Freyd cover $Gl(\Gamma)$ along $\Gamma$. Its objects are triple $(A,B,\alpha: A \to \Gamma(B))$ where $A\in U_0, B\in Ext(\mathbb{U})$. There is a projection map $p: Gl(\Gamma)\to Ext(\mathbb{U}) $ acting as $(A,B,\alpha: A \to \Gamma(B)) \mapsto B$ which is an AU functor.
By initiality we have a map $T: U_0 \to Gl(\Gamma)$. Again by initiality the diagram 
\[
\begin{tikzcd}
&Ext(\mathbb{U}) \\
Gl(\Gamma)\arrow[ur,"p"]]&\\
& U_0 \arrow[ul, "T"] \arrow[uu, "R" ]
\end{tikzcd}
\]
commutes. Hence $T$ acts as $A\mapsto (A, R(A)=A', \eta_A: A\to \Gamma(A') $. The map $\eta_A: A \to \Gamma(A')$ furnishes our required map. 
\end{proof}
\begin{remark}[Ingo]
Remark that this is quite curious as it would appear to say that: 
$U_0 \models A \Rightarrow A' \text{is provable}$
which seems false.
Indeed the analogous statement
\[
PA \vdash (\phi \to (PA\vdash \phi)
\]
is absolutely false. Take $\phi =Con(PA)$, then we conclude that if $PA$ were consistent it would prove its own consistency, hence would be inconsistent by Godel II.
We are saved however by the observation that $Con(PA)$ is not a formula in arithmetic type theory, but a sequent (as it uses a negation). Therefor this argument cannot be carried out internally. 
\end{remark}

\iffalse
\begin{thm}[Canonicity]
We have that $\iota$ induces an isomorphism $Hom_{\mathbb{U}_0'}(1_{\mathbb{U}_0'},\mathbb{N}') \cong \mathbb{N}$.
\end{thm}
\begin{proof} This may be done by hand. We have the \Jcomment{Why do we even need this??}
\end{proof}
\fi

The internal arithmetic universe $\mathbb{U}_0$ does not have power objects. However, the power object do exist, one level higher, in $U_0$.
Let $X:1 \to \mathbb{U}_0$ be a (global) object of $\mathbb{U}_0$. We have $\mathbb{U}_0(\bullet,X)$. To construct the power object of subobject $\mathcal{P}(X)$ we would first have to construct the monomorphism. The usual definition of monomorphism cannot be expressed in arithmetic type theory unfortunately, as it uses a universal quantifier one too many times. Fortunately, there is an alternate characterization: a map $i: Y\to X$ is a monomorphism if in the following diagram
\[
\begin{tikzcd}
Y\times_XY \arrow[d, "p_1"]\arrow[r, "p_2"] & Y \arrow[d]\\
Y \arrow[r] &X
\end{tikzcd}
\]
we have $p_1=p_2$. Hence the object of monomorphisms is constructed as $Mono(X)= \Sigma_{f:\mathbb{U}_0(\bullet,X)}p_1=p_2$. The subobjects may be similarly constructed by defining the relation $S_X$ on $Mono(X)$ given by $f\sim_Sg=\Sigma_{f,g: Mono(X)} Cod(f) \to Cod(g)$ and then taking the quotient. 

Given a subobject $1 \xrightarrow{A} \mathcal{P}1'$ there are two natural subobjects in $U_0$ associated to $A$: the equalizer
\[
\begin{tikzcd}
\ext{A=True'} \arrow[r] &1 \arrow[r, shift left=1.5ex,"A"] \arrow[r,swap, "True'"]& \mathcal{P}1'
\end{tikzcd}
\]
and the global sections $\Gamma(A)=\mathbb{U}'_0(1',A)$. These notions coincide:

\begin{prop} Let $1\xrightarrow{A}\mathcal{P}(1')$. The following subobjects are equal:
\[
\Gamma(A)=\mathbb{U}'_0(1',A)= \ext{A=True'}
\]
\end{prop}
\begin{proof}
Let's investigate the relation between $\Gamma(A)$ and $\ext{A=True'}$. Recall how $\Gamma(A)$ is constructed: 
Let $K_0$ denote the pullback
\[
\begin{tikzcd}
K_0 \arrow[r] \arrow[d] & \mathbb{U}_1(1',\bullet) \arrow[d] \\
Mono(\bullet,1') \arrow[r, " cod" ] & \mathbb{U}_0
\end{tikzcd}
\]
over $\mathcal{P}1'$ we have the commutative triangle
\[
\begin{tikzcd}
K_0 \arrow{dr} \arrow{rr} && 1 \arrow[dl,"True'"]\\
&\mathcal{P}1'&
\end{tikzcd}
\]

pullback this triangle along $1\xrightarrow{A} \mathcal{P}1'$ to obtain 
\[
\begin{tikzcd}
\Gamma(A) \arrow[dd,hook] \arrow[dr, hook] \arrow[rr]&&K_0 \arrow[dd] \arrow[dr,hook]&\\
&1\arrow[dl] \arrow[rr] && 1 \arrow[dl,"True'"]\\
1\arrow[rr, "A"]&&\mathcal{P}1' &
\end{tikzcd}
\]
where the front and back faces are pullback squares. We obtain a map $\Gamma(A) \to \ext{A=True'}$ by the universal property of the equalizer. 

We have a map $True': 1 \to K_0$. Compose this with the arrow $\ext{A=True'} \to 1$ to obtain a map $\ext{A=True'} \to K_0$. Together with the map $\ext{A=True'} \to 1$ this produces a map $\ext{A=True'} \to \Gamma(A)$ by the universal property of the pullback.
\end{proof}

\begin{prop}
\begin{enumerate}
    \item If $1 \xrightarrow{\mathbb{N}'} \mathcal{U}_0'$ is the formal natural numbers object, then there is an enumeration 
\[
e : \mathbb{N} \twoheadrightarrow \mathcal{P}'\mathbb{N}'
\]
of "all formulae with one variable free."
    \item The terminal object $1$ is projective. 
\end{enumerate}
\end{prop}
\begin{proof}
Recall that the initial arithmetic universe $U_0$ is constructed in three stages: take the initial Skolem theory $\Sigma_0$, take its category of predicates $Pred(\Sigma_0)$ and finally construct the ex/reg-completion.

The construction of $\Sigma_0'$ i complicated but essentially the morphisms $Mor(\Sigma_0')$ are given by primitive recursive functions. These are primitive recursively enumerable hence we have an epimorphism $\N \twoheadrightarrow Mor(\Sigma_0')$. Similarly, for $Mor(U_0')$.
Observe that we have a map $Im: Mor(U_0') \to Mono(U_0')$ which is defined by sending $f:A \to B$ to $im(f) \hookrightarrow B$.
Next, for a $\mathbb{N}'$ apply the quotient map $Mono(\bullet, \mathbb{N}')\twoheadrightarrow P(\mathbb{N}')$. Notice that after applying this quotient map it will be the identity on monomorphisms. We conclude that we have the required epimorphism $\N \twoheadrightarrow \mathcal{P}(\mathbb{N}')$.

Suppose we have an epimorphism $f:(Y_1,A_1) \to (Y_2,A_2)$ for equivalence relations $A_1,A_2$ on objects $Y_1,Y_2 \in Pred(\Sigma_0)$, and a map $x: 1 \to (Y_2,A_2)$. By construction of the exact/regular completion the map $x:1 \to (Y_2,A_2)$ lifts to a map $\tilde{x}: 1 \to Y_2 \twoheadrightarrow (Y_2,A_2)$. Take the pullback
\[
\begin{tikzcd}
(Y_1,A_1)\times_{(Y_2,A_2)}Y_2 \arrow[d, two heads] \arrow[r, two heads]& Y_2 \arrow[d,two heads]\\
(Y_1,A_1) \arrow[r, two heads] & (Y_2,A_2)
\end{tikzcd}
\]
where we used that epimorphisms are stable under pullback. The pullback $(Y_1,A_1)\times_{(Y_2,A_2)}Y_2$ is also of the form $(Y_3,A_3)$ for some equivalence relation $A_3$ on $Y_3\in Pred(\Sigma_0)$. Hence we have an epimorphism $Y_3 \twoheadrightarrow (Y_3,A_3)$; composing with $(Y_1,A_1)\times_{(Y_2,A_2)}Y_2 \twoheadrightarrow Y_2$ we get an epimorphism $Y_3 \twoheadrightarrow Y_2$ using that an epimorphism in an exact/completion of a category $C$ between objects in the image of $C$ yields an epimorphism in $C$. In $Pred(\Sigma_0)$ we know that every map factor through a split epimorphism followed by a monomorphism. We conclude that the map $Y_3 \twoheadrightarrow Y_2$ splits hence by composition we obtain a map $1 \to Y_3$ and hence a map $1 \to (Y_1,A_1)$.
\end{proof}
\iffalse
\begin{rmk}\Jcomment{These 'facts' need to be proved.} Here are three facts about $\tilde{\hat{\mathcal{S}}}\cong \mathcal{U}_0$ which will be needed:
\begin{enumerate}
    \item any object $X$ of $\tilde{\hat{\mathcal{S}}}$ with global support [i.e. $X \twoheadrightarrow 1$ is an epimorphism] can be enumerated by $\mathbb{N}$. This means that there is an epimorphism $\mathbb{N} \twoheadrightarrow X$.
    \item If $1 \xrightarrow{\mathbb{N}'} \mathcal{U}_0'$ is the formal natural numbers object, then there is an enumeration 
\[
e : \mathbb{N} \twoheadrightarrow \mathcal{P}'\mathbb{N}'
\]
of "all formulae with one variable free."
    \item The terminal object $1$ is projective. 
\end{enumerate}
\end{rmk}
\fi
Finally, we construct 
\[
\mathbb{N} \mapsto ( \iota: \mathbb{N} \to [1',\mathbb{N}'])
\]
where $\eta_{\mathbb{N}}=\iota$ is the map which assigns each natural number $n$ its "formal expression" $n'$. 
\begin{thm}[G\"odel's First Incompleteness Theorem]
If an arithmetic universe object $\mathbb{U}$ in $\mathcal{U}_{rec}$ is [syntactically] complete then it is the trivial $AU$-object [hence inconsistent].
\end{thm}

\begin{proof}
Cantor's Diagonal argument will now be imitated to prove G\"odel's Incompleteness Theorem in a special case. Form the pullback:
\[
\begin{tikzcd}
D \arrow[d,hook] \arrow[rrr] &&& 1 \arrow[d,hook,"false'"]\\
\mathbb{N} \arrow[r,"\Delta_{\mathbb{N}}"]& \mathbb{N} \times \mathbb{N}\arrow[r,"\iota \times e"]& \left[ 1' ,\mathbb{N}' \right] \times \mathcal{P}'\mathbb{N}'\arrow[r,"eval"]& \mathcal{P}'1'
\end{tikzcd}
\]

Since 1 is projective in the initial AU form the lift $n$ of $ D'$:
\[
\begin{tikzcd}
&\mathbb{N}\arrow[d,"e"]\\
1 \arrow[ur, dotted, "n"] \arrow[r, "D'"]& \mathcal{P}'\mathbb{N}' \\
\end{tikzcd}
\]
then
\[
\begin{tikzcd}
\ext{n\in D} \arrow[r,hook] \arrow[d,hook]&D \arrow[d,hook] \arrow[rrr]&&& 1 \arrow[d,hook,"false'"]\\
1\arrow[r,"n"]&\mathbb{N} \arrow[r,"\Delta_{\mathbb{N}}"]& \mathbb{N} \times \mathbb{N}\arrow[r,"\iota \times e"]& \left[1',\mathbb{N}'\right] \times \mathcal{P}'\mathbb{N}'\arrow[r,"eval"]& \mathcal{P}'1'
\end{tikzcd}
\]

is given by pullback. The subobject $\ext{n\in D}$ is the categorical incarnation of the G\"odel sentence $G$; asking whether this subobject factors through $1$ is equivalent to asking whether $G$ is provable.

Consider $\ext{n \in D}': 1 \hookrightarrow \mathcal{P}'1'$ , it follows that this is a pullback square by composition
\[
\begin{tikzcd}
\ext{n \in D} \arrow[d, hook] \arrow[r,hook]& 1\arrow[d, hook, "false'"]\\
1 \arrow[r]& \mathcal{P}'1'
\end{tikzcd}
\]
The bottom map is the composition $eval \circ \iota\times e \circ \Delta_{\nno} \circ n$; we claim this composition coincides with $\ext{n\in D}'$. Indeed $e\circ n= D'=R(D) $ by definition of $n$. By construction also $\iota(n)=R(n)=n'$ hence the composition is $\ext{n'\in D'}=\ext{n\in D}'$.

If $\ext{n\in D}$ equals $0=False$ then $\ext{n\in D} $ is the pullback of $false'$ along $false'$,hence we conclude $\ext{n\in D}$ equals $1=True$. This contradicts the consistency of $\mathcal{U}_0$. On the other hand, if $\ext{n\in D}$ equals $1=True$ then $\ext{n\in D} $ is the pullback of $True'$ along $False'$, hence we conclude $\ext{n\in D}$ equals $0=False$, contradiction.
We conclude that the G\"odel sentence $\ext{n\in D}$ is neither $0$ or $1$.
\end{proof}

We are now ready to see the proof of Godel's second incompleteness theorem for PRA. Let $U_0, U_0'$ be as before. The arithmetic universe $U_0$ proves the consistency of $U_0'$ if $\ext{True'=False'}\hookrightarrow 1$ equals the minimal subobject $0\hookrightarrow 1$.
\begin{thm}[G\"odel's Second Incompleteness Theorem] Assume that $U_0$ is consistent. Then the subobject $\ext{True'=False'}\hookrightarrow 1$ does not equal the minimal subobject $0\hookrightarrow 1$ in $U_0$.
\end{thm}
\begin{proof}
We have the diagram
\[
\begin{tikzcd}
\ext{n \in D} \arrow[d, hook] \arrow[r,hook]& 1\arrow[d, hook, "false'"]\\
1 \arrow[r, "\ext{n\in D}' " ]& \mathcal{P}'1'
\end{tikzcd}
\]
we also have 
\[
\begin{tikzcd}
\ext{n\in D} \arrow[r, hook] \arrow[d,hook] & 1 \arrow[d, hook, "true' "]\\
1 \arrow[r, "\ext{n\in D}' "]& \mathcal{P}'1'
\end{tikzcd}
\]
this follows from the canonical map $\ext{n\in D} \to \Gamma(\ext{n\in D}')=\ext{\ext{n\in D}'=True'}$.
In turn the above implies that
\[
\begin{tikzcd}
\ext{ n\in D} \arrow[r,hook] &1 \arrow[r, shift left=1.5ex, "false' "] \arrow[r,swap,  "true' "]&
\mathcal{P}'1'
\end{tikzcd}
\]
commutes. We will show that it is also an equalizer.
Let
\[
\begin{tikzcd}
Z \arrow[r,"!_Z"] &1 \arrow[r, shift left=1.5ex,"false'"] \arrow[r,swap, "true'"]& \mathcal{P'}1'
\end{tikzcd}
\]
be a commutative diagram. The object $P1$ is an internal poset with least element $false'$ and largest element $true'$. Therefore the externalization $\hom(Z, \mathcal{P'}1')$ is also a poset with bottom $false'\circ\ !_Z$ and top $true'\circ\ !_Z$. Also $false'\circ\ !_Z \leq \ext{n\in D}' \leq true'\circ\ !_Z = false'\circ\ !_Z$. Thus $\ext{n\in D}' = true'\circ\ !_Z = false'\circ\ !_Z$. By the above pullback $!_Z$ factors through $\ext{ n\in D}$.

If $\ext{n\in D}$ equals $0=False$ then $\ext{n\in D} $ is the pullback of $false'$ along $false'$,hence we conclude $\ext{n\in D}$ equals $1=True$. This contradicts the consistency of $\mathcal{U}_0$. 
It follows that $\ext{n\in D}=\ext{True'=False'}$ is not $0$ and the consistency of $\mathcal{U}_0'$ is not provable.
\end{proof}

\iffalse
\begin{remark}Notice that in the proof the map $1 \xrightarrow{n} \nno$ is a lift of $1\xrightarrow{D'}\mathcal{P}'\nno'$ which implictly refers to yet another level of self-reference: one repeats the construction and considers inside $U_0'$ a further $U_0''$ and defines
\[
\begin{tikzcd}
D' \arrow[d,hook] \arrow[rrr] &&& 1 \arrow[d,hook,"false'"]\\
\mathbb{N}' \arrow[r,"\Delta_{\mathbb{N}'}"]& \mathbb{N}' \times \mathbb{N}'\arrow[r,"\iota' \times e'"]& \left[ 1' ,\mathbb{N}'' \right] \times \mathcal{P}''\mathbb{N}''\arrow[r,"eval"]& \mathcal{P}''1''
\end{tikzcd}
\]
Potentially one can go on and on, fortunately this is as far as we need to go.
\end{remark}
\fi
%\include{Chapters/LastCommutativeDiagram}
%\include{Chapters/Letter}
\section{Lob's Theorem}
This is a sketch of a proof of Lob's theorem in Arithmetic Universes. For now, the reader can assume that $U=U_0$ is the initial arithmetic universe. 

To give a proof of Lob's theorem we will need to be able to state Lob's Theorem.
\begin{defn}
We say that $\phi$ implies $\psi$ in $U$, written $U \models \phi \vdash \psi$, if there is an inclusion of subobjects $\phi \hookrightarrow \psi$. We interpret $\phi$ as the judgement $\top \vdash \phi$. 
\end{defn}
Notice that as defined above, a judgement $\phi \vdash \psi$ is not itself a proposition. It is therefore not immediately apparent how to interpret the (conceptually distinct) sentences like $(\phi \rightarrow \psi) \rightarrow \chi$ or $\phi \rightarrow (\psi \rightarrow \chi)$. 

An arithmetic universe is typically not cartesian closed, so the usual way of talking about implication will not work here. 
Let $U$ be an arithmetic universe, let $\phi,\psi \hookrightarrow 1$ be monomorphisms/propositions. Let $U[\phi \to \psi]=U[\phi \leq \psi]$ be the classifying AU of $\phi\leq \psi$. We have an adjoint pair of AU functors
\[
\begin{tikzcd}
\quad U
\arrow[r, "i^{\ast}"{name=F}, bend left=25] &
\quad U[\phi \leq \psi]
\arrow[l, "i_{\ast}"{name=G}, bend left=25]
%--- Adjunction Symbol
\arrow[phantom, from=F, to=G, "\dashv" rotate=-90, shift left=1.5ex]
\end{tikzcd}
\]
Moreover, we also know $U[\top \leq \phi]=u[\phi]=U/\phi$, we may then identify $i^{\ast}$ with the pullback along $i: \phi \hookrightarrow 1$
\begin{defn}
Given two judgements $\phi \vdash \psi$ and $\sigma \vdash \tau$ we interpret implication of judgements $(\phi \vdash \psi) \rightarrow (\sigma \vdash \tau)$ as the statement that there is a morphism of subobjects $i^*(\sigma) \hookrightarrow i^*(\tau)$ in $U[\phi \leq \psi]$.
\end{defn}
\begin{defn}
We say $U\models$ $\ulcorner U \models \phi \vdash \psi \urcorner$ if the subobject $Hom(\phi', \psi') \hookrightarrow 1$ is inhabited.
\end{defn}
We will also write $U \models \square (\phi \vdash \psi) $ for $U\models$ $\ulcorner U \models \phi \vdash \psi \urcorner$.
\begin{thm}[L\"ob] Let $U$ as above, let $\phi \hookrightarrow 1$ be a given proposition (=monomorphism).
Then $U \models \square \phi \vdash \phi $ implies $U\models \phi$.
\end{thm}

We will need to verify a number of properties of the implication as defined as above.

\begin{prop} \label{A1}
If $U\models \phi$ then $U\models \square \phi $
\end{prop}
\begin{proof} Immediate from lemma \ref{FreydMap}.
\end{proof}

\begin{prop}\label{A2}
$U\models\square \phi  \vdash \square (\square \phi)  $
\end{prop}
\begin{proof}
Apply the functor $\square=\Gamma \circ R$ to $\phi \to \square \phi $.
\end{proof}

\begin{prop}[Modus Ponens] \label{ModusPonens}
$U \models \phi $ and $U \models \phi \vdash \psi $ then $U \models \psi$.
\end{prop}
\begin{proof}
Immediate from the definitions. 
\end{proof}
\begin{prop}[Internal Modus Ponens] \label{A3}
$U \models \ulcorner \ulcorner U \models \phi \vdash \psi \urcorner \rightarrow ( \square \phi \vdash \square \psi ) \urcorner$.
\end{prop}       
\begin{proof}
Let $U[Hom(\phi',\psi')]=U/Hom(\phi',\psi')$ denote the classifying AU of the proposition $Hom(\phi',\psi')$.
Write out to see that $\square(\phi)=\Gamma(\phi')=Hom(1',\phi')$. We have an evaluation map $ev: Hom(\phi',\psi')\times Hom(1',\phi') \to Hom(1',\psi')$. Remark that $i^{\ast}(\chi)=Hom(\phi',\psi') \times \chi$.

We want to prove that there is an arrow $i^{\ast}(\square \phi) \to i^{\ast}(\square \psi)$. 
We have the following diagram   
\[
\begin{tikzcd}
Hom(\phi',\psi')\times \square(\phi) \arrow[dr] \arrow[rr, "<id{,}ev>"] & & Hom(\phi',\psi')\times \square(\psi)\arrow[dl]\\
&Hom('\phi', \psi') & 
\end{tikzcd}
\]
Hence there is an arrow $i^{\ast}(\square \phi) \to i^{\ast}(\square \psi)$. 
\end{proof}
\begin{prop}\label{B1}
$U \models \phi \vdash \psi $ and $ U \models \psi \vdash \chi $ then $U\models \phi \vdash\chi$
\end{prop}
\begin{proof}
Immediate.
\end{proof}
\begin{prop}\label{B2}
$U\models \phi \vdash \psi$ and $U\models \phi \rightarrow (\psi \vdash \chi)$ implies $U\models \phi \vdash \chi$
\end{prop}
\begin{proof}
Consider the classifying arithmetic universe $U[\phi]=U/\phi$. By assumption we have a morphism $a: \phi \to \psi$ and a morphism $b: i^{\ast}\psi=\phi \times \psi \to i^{\ast} \chi=\phi \times \chi $ over $id_{\phi}: \phi \to \phi$.
The composition $p_2\circ b \circ <id_{\phi},a>: \phi \to \chi$ furnishes the required morphism. Here $p_2$ is the projection on the second coordinate. 
\end{proof}

We have the following lemma, the AU-incarnation of the diagonalization lemma. It is reminiscent of Lawvere's fixed point theorem. 
\begin{defn}
Let $f: X \to X$ be a map. Let $x:1 \to X$ be a global point. We say $x$ is a fixed point of $f$ if $f(x):=f \circ x =x.$
\end{defn}
\begin{lem} \label{DiagonalizationLemma}
Let $T: \mathcal{P}1'\to \mathcal{P}1'$ be a map in $U$. Then $T$ has a fixed point.
\end{lem}
\begin{proof}
Let $g$ be the composition
\[
\begin{tikzcd}
\mathbb{N}\times \mathbb{N} \arrow[r,"i \times e"] & {[1', \mathbb{N}']} \times \mathcal{P}\mathbb{N}' \arrow[r,"eval"]& \mathcal{P}1' \arrow[d,"T"] \\
\mathbb{N} \arrow[u,"\Delta"] \arrow[rr,"g"] && \mathcal{P}1' 
\end{tikzcd}
\]
Construct $E$ as the pullback 
\[
\begin{tikzcd}
E \arrow[d, hook] \arrow[r] & 1 \arrow[d, "True'"]\\
\mathbb{N} \arrow[r] & \mathcal{P}1'
\end{tikzcd}
\]
By projectivity of $1$ we have a lift 
\[
\begin{tikzcd}
& \mathbb{N} \arrow[d,"e", two heads]\\
1 \arrow[ur, "n", dashed] \arrow[r,"E'"] & \mathcal{P}\mathbb{N}'
\end{tikzcd}
\]
where $E'=R(E), R: U \to Ext(U')$.
Consider the composition 
\[
\mathbb{N} \cong \mathbb{N}\times 1 \xrightarrow{Id\times n} \mathbb{N} \times \mathbb{N}\xrightarrow{f} \mathcal{P}1'
\]
where $f$ denotes the morphism
\[
\mathbb{N} \times \mathbb{N} \xrightarrow{i\times e} [1',\mathbb{N}']\times \mathcal{P}\mathbb{N}'\xrightarrow{eval}\mathcal{P}1
\]
This is equal to $g$.
Finally, the claim is that $g(n)$ is a fixpoint for $T$.
By construction we have $f(t,n)=g(t)$. Therefore, $T(g(n))=T(f(n,n))=g(n)$ by definition of $g$ as the composition in the first square. 
\end{proof}

Let $\phi \hookrightarrow 1$ be a proposition. Consider the operator $T_{\phi}: \mathcal{P}1'\to \mathcal{P}1'$ defined as $T_{\phi}(\beta)=RHom(\square \beta,R(\phi)) $.
Apply the diagonalization lemma to obtain the Lob sentence $L_\phi$. It satisfies $RHom(R\square(L_{\phi}),R(\phi))=L_{\phi}$.

Recall our formulation of Lob's theorem:
\begin{thm} Let $U$ as above, let $\phi \hookrightarrow 1$ be a given proposition (=monomorphism).
Then $U \models (U\models \phi ) \rightarrow \phi $ implies $U\models \phi$.
\end{thm}
\begin{proof}
\begin{enumerate}
    \item \label{Step1} Start with $L=L_{\phi}:1 \to \mathcal{P}1'$, the Lob sentence constructed using Lemma \ref{DiagonalizationLemma} using the operator $T_{\phi}$ acting as $T_{\phi}(\beta)=RHom(\square \beta),R(\phi)): \mathcal{P}1'\to \mathcal{P}1'$. We have the equivalence $T_{\phi}(L)=L$.
    \item \label{Step2} By assumption we have $U \models \square \phi \vdash \phi$.
    \item \label{Step3}  $ U\models (\name{U\models L} \vdash \phi) \to (\name{U \models L} \vdash \square \phi)$ by proposition \ref{A3}.
    \item \label{Step4} $U\models \name{U \models L} \to (\name{U \models \name{ U\models L}} \to \name{U\models \phi}$ by steps \ref{Step1},\ref{Step3} and proposition \ref{B1}
    \item \label{Step5} $U\models \name{U\models L} \to \name{U\models \name{U \models \phi}}$ by proposition \ref{A2}.
    \item \label{Step6} $U \models \name{ U \models L} \to \name{U \models \phi}$ by steps \ref{Step4}, \ref{Step5} and proposition \ref{B2}.
    \item \label{Step7} $U\models \name{U \models L} \to \phi $ by steps \ref{Step2}, \ref{Step6} and proposition \ref{B1}.
    \item \label{Step8} $U\models \name{U\models \name{U \models L} \to \phi }$ by Step \ref{Step7} and proposition \ref{A1}.
    \item \label{Step9} $U \models \name{U \models L}$ by step \ref{Step1} and proposition \ref{ModusPonens} [Modus Ponens].
    \item \label{Step10} $U \models \phi$ by steps \ref{Step7}, \ref{Step9} and proposition \ref{ModusPonens} [Modus Ponens].
\end{enumerate}
\end{proof}
\begin{remark}
As arithmetic universes are not in general cartesian closed there is no one notion of implication. In many ways the above discussion merely evaded this issue by looking mostly at entailment. It would be of interest to investigate novel implication concepts and ascertain whether or not Lob's theorem may be proved of them.
\end{remark}

\bibliographystyle{apalike}
\bibliography{BibliographyAU}

\begin{thebibliography}{}

\bibitem[Blechschmidt, 2017]{Blechschmidt2017}
Blechschmidt, I. (2017).
\newblock {\em Using the internal language of toposes in algebraic geometry}.
\newblock PhD thesis.
\newblock Advisor: Marc Nieper-Wißkirchen.

\bibitem[Hazratpour and Vickers, 2018]{hazratpour18}
Hazratpour, S. and Vickers, S. (2018).
\newblock Fibrations of {AU}-contexts beget fibrations of toposes.
\newblock {\em arXiv preprint arXiv:1808.08291}.
\newblock Submitted to Theory and Application of Categories (TAC).

\bibitem[Hofmann, 1995]{hofmann95}
Hofmann, M. (1995).
\newblock {\em Extensional constructs in intensional type theory}.
\newblock PhD thesis, University of {Edinburgh}.
\newblock Advisor: D. Sannella.

\bibitem[Johnstone, 2003]{Elephant}
Johnstone, P. (2003).
\newblock {\em Sketches of an Elephant: A Topos Theory Compendium}.

\bibitem[Maietti, 2003]{maietti03}
Maietti, M.~E. (2003).
\newblock Joyal's arithmetic universes via type theory.
\newblock {\em Electronic Notes in Theoretical Computer Science}, 69:272--286.

\bibitem[Maietti, 2005]{maietti05b}
Maietti, M.~E. (2005).
\newblock Reflection into models of finite decidable {FP}-sketches in an
  arithmetic universe.
\newblock {\em Electronic Notes in Theoretical Computer Science}, 122:105--126.

\bibitem[Maietti, 2010]{maietti10a}
Maietti, M.~E. (2010).
\newblock Joyal's arithmetic universe as list-arithmetic pretopos.
\newblock {\em Theory and Applications of Categories}, 24:39--83.

\bibitem[Martin-L{\"o}f, 1975]{martin-lof75}
Martin-L{\"o}f, P. (1975).
\newblock An intuitionistic theory of types: Predicative part.
\newblock In {\em Studies in Logic and the Foundations of Mathematics},
  volume~80, pages 73--118. Elsevier.

\bibitem[Martin-Lof, 1998]{pml:int-th-type}
Martin-Lof, P. (1998).
\newblock {\em An Intuitionistic Theory of Types}.
\newblock In G.Sambin and Jan.Smith's Twenty Five Years of Constructive Type
  Theory. Oxford Logic Guides, Clarendon press, Oxford.

\bibitem[Martin-L{\"o}f and Sambin, 1984]{martin-lof84}
Martin-L{\"o}f, P. and Sambin, G. (1984).
\newblock {\em Intuitionistic type theory}, volume~1 of {\em Studies in Proof
  Theory}.
\newblock Bibliopolis Naples.

\bibitem[Morrison, 1996]{morrison96}
Morrison, A. (1996).
\newblock Reasoning in arithmetic universes.
\newblock Master's thesis, University of {London} - {Imperial} {College} of
  {Science}, {Technology} and {Medicine}.
\newblock Advisor: S. Vickers.

\bibitem[Palmgren and Vickers, 2007]{palmgrenvickers2007}
Palmgren, E. and Vickers, S.~J. (2007).
\newblock Partial horn logic and {C}artesian categories.
\newblock {\em Ann. Pure Appl. Logic}, 145(3):314--353.

\bibitem[{Univalent Foundations Program}, 2013]{HoTT2013}
{Univalent Foundations Program}, T. (2013).
\newblock {\em Homotopy type theory---univalent foundations of mathematics}.
\newblock The Univalent Foundations Program, Princeton, NJ; Institute for
  Advanced Study (IAS), Princeton, NJ.

\bibitem[Vickers, 1999]{vickers99a}
Vickers, S. (1999).
\newblock Topical categories of domains.
\newblock {\em Mathematical Structures in Computer Science}, 9(5):569--616.

\bibitem[Vickers, 2017]{vickers17}
Vickers, S. (2017).
\newblock Arithmetic universes and classifying toposes.
\newblock {\em Cahiers de {Topology} et {G\'eom\'etrie} {Diff\'erentielle}
  {Cat\'egoriques}}, 58(4):213--248.

\end{thebibliography}

\end{document}